\documentclass[12pt,twoside]{amsart}
\usepackage{amscd,amsfonts,amsmath,amssymb,amsthm,latexsym}
\usepackage{enumerate,array,srcltx,amsmath}
\usepackage{leqno}
\usepackage[dvips]{graphicx}
\usepackage{a4}
\usepackage{fancyhdr}
\usepackage{caption}
\usepackage{subcaption}
\usepackage{mathrsfs}
\usepackage[driverfallback=pdfmx,,colorlinks,bookmarksopen,bookmarksnumbered,citecolor=blue, urlcolor=blue]{hyperref}
\usepackage{color}

\definecolor{darkblue}{rgb}{0,0,0.545098}
\definecolor{darkgreen}{rgb}{0,0.392157,0}

\newtheorem{theorem}{Theorem}[section]
\newtheorem{lemma}[theorem]{Lemma}
\newtheorem{proposition}[theorem]{Proposition}
\theoremstyle{definition}

\theoremstyle{remark}
\newtheorem{remark}[theorem]{Remark}
\numberwithin{equation}{section}

\newcommand{\R}{\mathbb{R}}
\newcommand{\C}{\mathbb{C}}

\newcommand{\func}{\operatorname}
\newenvironment{pf}{\vspace{0.5em}\noindent\textbf{Proof.} }{\quad \hfill $\Box$ \\ \vspace{0.5em}\\}

\newcommand{\AF}{\mathcal{A}_{\mathrm{F}}}
\newcommand{\AC}{\mathcal{A}_{\mathrm{C}}}

\newcommand{\AbF}{\mathcal{A}^B_{\mathrm{F}}}
\newcommand{\AbC}{\mathcal{A}^B_{\mathrm{C}}}
\newcommand{\cH}{\mathcal{H}}
\newcommand{\D}{\mathcal{D}}

\newcommand{\VF}{V_{\mathrm{F}}}
\newcommand{\VC}{V_{\mathrm{C}}}
\newcommand{\WF}{W_{\mathrm{F}}}
\newcommand{\WC}{W_{\mathrm{C}}}

\newcommand{\EF}{E_{\mathrm{F}}}
\newcommand{\EC}{E_{\mathrm{C}}}
\newcommand{\EcF}{\mathcal{E}_\mathrm{F}}
\newcommand{\EcC}{\mathcal{E}_\mathrm{C}}

\newcommand{\LF}{L_{\mathrm{F}}}
\newcommand{\LC}{L_{\mathrm{C}}}
\newcommand{\cLF}{\mathcal{L}_\mathrm{F}}
\newcommand{\cLC}{\mathcal{L}_\mathrm{C}}

\newcommand{\rF}{\rho_{\mathrm{f}}}
\newcommand{\rC}{\rho_{\mathrm{c}}}
\newcommand{\vrF}{\varrho_{\mathrm{f}}}
\newcommand{\vrC}{\varrho_{\mathrm{c}}}

\newcommand{\F}{\mathscr{F}}

\def\Re{\func{Re}}

\title[SLV EQUATION WITH HEAT CONDUCTION]{On the Cauchy problem of the standard linear solid model with heat conduction: Fourier versus Cattaneo}
\author{Marta Pellicer$^{1}$}
\thanks{$^{1}$Dpt. d'Inform\`atica, Matem\`atica Aplicada i Estad\'{\i}stica (EPS), Universitat de Girona, Campus de Montilivi, 17071 Girona, Catalunya (Spain), e-mail: martap@imae.udg.edu}
\author{Belkacem  Said-Houari$^{2}$}
\thanks{$^{2}$Department of Mathematics, College of Sciences, University of
Sharjah, P. O. Box: 27272, Sharjah, United Arab Emirates,
Email:bsaidhouari@gmail.com}
\makeatother

\date{\today}

\begin{document}

\maketitle
\begin{abstract}
In this paper, we consider the standard linear solid model in $\mathbb{R}^N$ coupled, first, with the Fourier law of heat conduction and, second, with the Cattaneo law. First, we give the appropriate functional setting to prove the well-posedness of both models under certain assumptions on the parameters (that is, $0<\tau\leq \beta$). Second, using the energy method in the Fourier space,  we obtain the optimal decay rate of a norm related to the solution both in the Fourier and the Cattaneo heat conduction models under the same assumptions on the parameters. More concretely, we  prove that, when $0<\tau<\beta$, the model with heat conduction has the same decay rate as the Cauchy problem without heat conduction (see \cite{PellSaid_2016}) both under the Fourier and Cattaneo heat laws. Also, we are  able to see that the difference between using a Fourier or Cattaneo law in the heat conduction is not in the decay rate, but in the fact that the Cattaneo coupling exhibits the well-know regularity loss phenomenon, that is, Cattaneo model requires a higher regularity of the initial data for the solution to decay.
When $0<\tau=\beta$ (that is, when the dissipation comes through the heat conduction) we  still have asymptotic stability in both heat coupling models, but with a slower decay rate. As we prove later, such stability is not possible in the absence of the heat conduction.  We also prove the optimality of the previous decay rates for both models by using the eigenvalues expansion method. Finally, we complete the results in \cite{PellSaid_2016} by showing how the condition $0<\tau< \beta$ is not only sufficient but also necessary for the asymptotic stability of the problem without heat conduction.


\end{abstract}


\section{Introduction}\label{Introduction}
This paper deals with well-posedeness, stability  and decay rates of a class of vibration problems modeling the so-called \emph{standard linear solid model} of viscoelasticity coupled with heat conduction. Without considering the heat conduction effect,   Gorain \& Bose \cite {Gorain_Bose_1998} (see also   \cite{Bose_Gorai_1998}) derived the equation
\begin{equation}\label{Bose_Problem}
\tau u_{ttt}+u_{tt}-a^{2}\Delta u-a^{2}\beta \Delta u_{t}=0,
\end{equation}
as a more realistic approach to model flexible structural systems possessing internal damping (see more details below). By constructing a suitable Lyapunov functional, Gorain \cite{Gorain_2010} investigated the initial and boundary value problem associated to (\ref{Bose_Problem}) an showed an exponential decay of the energy
\begin{equation}\label{Energy_Bounded_domain}
E(t)=\frac{1}{2}\int_\Omega\left\{|\tau u_{tt}+u_t|^2+a^2|\tau\nabla u_t+\nabla u|^2+a^2\tau(\beta-\delta)|\nabla u_t|^2\right\}dx,
\end{equation}
provided that
\begin{equation}\label{condition}
0<\tau<\beta.
\end{equation}
The study of problem (\ref{Bose_Problem}) has been extended to an abstract setting in \cite{FLP_2011} where the authors studied  the regularity of mild and strong solutions of an abstract mathematical model of a flexible space structure under appropriate initial conditions.

Equation (\ref{Bose_Problem}) can be obtained by considering the following relation
 \begin{equation}\label{Constitutive_Equation}
\sigma+\tau\sigma^{\prime}=E(e+\beta e^\prime)
\end{equation}
where $\sigma$ is the stress, $e$ is the strain, $E$ is the Young modulus of the elastic structure and the constants $\tau$ and $\beta$ are small and satisfying \eqref{condition}. If $\tau=\beta=0$ in (\ref{Constitutive_Equation}), then we get $\sigma=Ee$ which indicates that  the vibrations are governed by the free wave equation
\begin{equation*}
u_{tt}-a^{2}\Delta u=0.
\end{equation*}
Under the assumption \eqref{condition}, we obtain \eqref{Bose_Problem},  the so-called standard linear solid model or, also, standard linear model of viscoelasticity
(see \cite{Gorain_2010} or \cite{PellSaid_2016} for more details on the modelling). Surprisingly, this equation is also known as the Moore-Gibson-Thompson equation and arises in Acoustics as an alternative model to the well-know Kuznetsov equation to describe the vibrations of thermally relaxing fluids (see \cite{KatLasPos_2012}, for instance).

This problem has been recently studied by several authors in a bounded domain (apart from the references above, see also \cite{Kaltenbacher_2011},  \cite{Trigg_et_al} or \cite{P-SM-2019} and the references therein, among others) and also in $\R^N$ in \cite{PellSaid_2016} and \cite{Racke_Said_2019} . In these references we can see that the relation \eqref{condition} turns out to be critical for the stability of this problem. Roughly speaking, it has been known that, in the bounded domain problem, the strongly continuous semigroup that describes the dynamics of this problem is exponentially stable provided the dissipative condition \eqref{condition} is fulfilled  and that the system exhibits a chaotic behaviour when $\tau>\beta>0$ (see \cite{Conejero} for this last result). Several works on the Moore-Gibson-Thompson equation with memory have also been published (see, for instance, \cite{LasieckaPata} and the references therein).

Recently, the authors in \cite{PellSaid_2016} showed the well-posedness of the corresponding Cauchy problem, as well as the optimal decay rate of its solutions (using the energy method in the Fourier space to build an appropriate Lyapunov functional, together with the eigenvalues expansion method). The authors of the   recent paper  \cite{Racke_Said_2019} investigated a nonlinear version of \eqref{Bose_Problem} and proved a global existence and decay results for the solution under the smallness assumption on the initial data.

On the other hand, a well-known model for the heat conduction is the Fourier law, that is:
$$\theta_t+\gamma \nabla\cdot q=0$$
and
$$q+\kappa\nabla\theta=0$$
for $\gamma,\beta>0$, where $\theta$ and $q$ stand for the temperature difference and heat flux vector, respectively. Combining the previous two equations we obtain the parabolic heat equation:
$$\theta_t-\gamma\kappa\Delta\theta=0$$
that allows an infinite speed for thermal signals. To overcome this drawback in the Fourier law, we can consider the Cattaneo law (also known as Maxwell-Cattaneo law)
of heat conduction  that differs from the Fourier one by the presence of the relaxation term $q_{t}$ as
\begin{equation*}
\tau _{0}q_{t}+q+\kappa  \nabla\theta =0
\end{equation*}
with $\tau_0>0$, which now models the heat propagation equation as a damped wave equation. This phenomenon is known as \textsl{second sound} effect and it is experimentally observed in materials at a very low temperature, where heat seems to propagate as a thermal wave, which is the reason for this name (see \cite{ABFRS_2012}, \cite{Apalara_et_al} or \cite{Montanaro}, and the review in \cite{Racke_book}). The parameter $\tau_0>0$ is the relaxation time of the heat flux (usually small compared to other parameters) and arises due to the delay in the response of the heat flux to the temperature gradient.

Recently, the authors in \cite{ABFRS_2012} considered equation (\ref{Bose_Problem}) coupled with the Fourier law of heat conduction in bounded domain. Namely, they studied the system
\begin{eqnarray}\label{Rivera_Model}
\left\{
\begin{array}{l}
\tau u_{ttt}+u_{tt}-a^2\Delta u-a^2\beta\Delta u_t+\eta\Delta\theta=0,\vspace{0.2cm}\\
\theta_t-\Delta\theta-\tau\eta\Delta u_{tt}-\eta\Delta u_t=0,
\end{array}
\right.
\end{eqnarray}
where $x\in\Omega$ (a bounded domain of $\R^N$) and $t\in(0,\infty)$ and $\eta\geq 0$ being the coupling coefficient. Observe that when $\eta=0$ we recover the standard linear solid model of viscoelasticity \eqref{Bose_Problem}. The authors proved that under the assumption \eqref{condition} the system (\ref{Rivera_Model}) is well-posed and, also, they showed that the total energy $$E(t)+\Vert\theta(t)\Vert_{L^2(\Omega)}^2$$ decays exponentially, where $E(t)$ is given in (\ref{Energy_Bounded_domain}). Although thermoelasticity and thermoviscoelasticity have been widely treated in the literature, to our knowledge this is the only reference to the standard linear solid model coupled with heat conduction. Previous thermoelastic and thermoviscoelastic models  have been studied in the literature from the point of view of the stability of their solutions. These models are, for instance, the strongly damped wave equation (or wave equation with Kelvin-Voigt damping) with heat conduction, using either the Fourier or the Cattaneo law of heat conduction (see \cite{Misra2014} as a reference for the former, or \cite{Alves2016} or \cite{Racke_2002} as references for the latter). Or also the models of classical thermoelasticity (that is, wave equation coupled with heat conduction), where both Fourier and Cattaneo laws have been used for the heat coupling (see \cite{Carlson}, \cite{LiuZheng} or \cite{Racke_2002}, among many others, as references for the asymptotic stability of these models). Finally, we can find works on the asymptotic stability of thermoelastic plates, with also either Fourier or Cattaneo laws being used (see \cite{MRRackeJDE96} and \cite{RackeUeda}, respectively, or \cite{Said_2013} for a comparison between both approaches).


Our goal in this paper is to consider the equation (\ref{Bose_Problem}) coupled either with the Fourier or the Cattaneo law of heat conduction in $\R^N$. That is, we want to consider the following system:
\begin{eqnarray}\label{Rivera_Model_Cattaneo}
\left\{
\begin{array}{l}
\tau u_{ttt}+u_{tt}-a^{2}\Delta u-a^{2}\beta \Delta u_{t}+\eta\Delta \theta =0,\vspace{0.2cm}\\
\theta _{t}+\gamma \nabla\cdot q -\tau\eta \Delta u_{tt}-\eta \Delta u_{t}=0,\vspace{0.2cm}\\
 \tau_0q_t+q+\kappa \nabla\theta=0,
\end{array}
\right.
\end{eqnarray}%
where $x\in \mathbb{R}^N,\,t\geq 0$, $\tau, \gamma,\kappa,\delta > 0$ and $\tau_0,\eta\geq 0$, for both the Fourier model (i.e., $\tau_0=0$) and the Cattaneo one (i.e., $\tau_0>0$). We show that, when $0<\tau< \beta$ and also when $0<\tau=\beta$, the corresponding systems are well-posed in the appropriate functional setting. Also we  prove their asymptotic stability giving the decay of certain norms related to the norm of the solution of the corresponding Cauchy problems. We compare the obtained decay rates between them and also with the ones of the problem without heat conduction, in order to see how the heat conduction term and its type affects the asymptotic behaviour of the solutions. We would like to remark that in the present paper we  also discuss the case $\tau=\beta$ (not only the case $0<\tau<\beta$ as in our previous work \cite{PellSaid_2016}), both in the Fourier and Cattaneo case, which turns out to be important when we add heat conduction to the standard linear solid model. This is because the only damping in this case comes from the heat conduction. In this paper we see that, while when $\eta=0$ (no heat conduction) the problem for $\tau=\beta$ becomes unstable (to our knowledge this was only proved through numirical method so far).   When $\eta>0$ we are able to show the asymptotic stability both for the Fourier and Cattaneo problems.

The asymptotic stability results we prove in the following sections are the following decay rates. First, we  show the following decay rate for the Fourier model under the condition \eqref{condition}:
\begin{equation*}\label{Main_estimate_Fourier_model}
\left\Vert \nabla^{k}\VF\left( t\right) \right\Vert _{L^{2}(\R^N)}\leq  C(1+t)^{-N/4-k/2}\left\Vert \VF^0 \right\Vert _{L^{1}(\R^N)}+Ce^{-ct}\left\Vert \nabla^{k}\VF^0 \right\Vert _{L^{2}(\R^N)},
\end{equation*}
for any $t\geq 0$ and $0\leq k\leq s$, where $\VF(x,t)=(\tau u_{tt}+u_t,\nabla (\tau u_t+u),\nabla u_t,\theta)^T(x,t)$ and $\VF^0=\VF(x,0)$ (see Theorem \ref{Decay_Second_system}). On the other hand, we investigate the coupling with the Cattaneo law and prove that the decay estimate when \eqref{condition} is:
\begin{equation*}\label{Main_estimate_Cattaneo_law}
\left\Vert \nabla^{k}\VC\left( t\right) \right\Vert _{L^{2}(\R^N)}\leq  C(1+t)^{-N/4-k/2}\left\Vert \VC^0 \right\Vert _{L^{1}(\R^N)}+C(1+t)^{-\ell/2}\left\Vert \nabla^{k+\ell}\VC^0 \right\Vert _{L^{2}(\R^N)},
\end{equation*}
for any $t\geq 0$ and $ 0\leq k+\ell\leq s$, where $\VC(x,t)=(\tau u_{tt}+u_t,\nabla (\tau u_t+u),\nabla u_t,\theta,q)^T(x,t),$ $\VC^0=\VC(x,0)$ (see Theorem \ref{Decay_Second_system_Cattaneo}). Thus, the exponent of the decay rate under the assumption \eqref{condition} is the same for both types of heat coupling and, actually, the same as the one obtained in the Cauchy problem for the model without heat conduction (see Theorem 3.6 in \cite{PellSaid_2016}). In particular, adding the heat conduction does not produce a faster decay rate. The difference between both heat laws when $0<\tau<\beta$ (and also with the model without heat conduction) is that the Cattaneo model exhibits the decay property of regularity-loss type. This phenomenon is well-known in other models (see for instance \cite{HKa06} and \cite{IHK08}), and might be usually expected in the presence of Cattaneo law of heat conduction (see \cite{Said_2013}, for instance).

In the case $\tau=\beta$, we see that the decay rates are slightly different from the previous ones. For the Fourier model we have
\begin{equation*}\label{Main_estimate_Fourier_model_tau=beta}
\left\Vert \nabla^{k}\WF\left( t\right) \right\Vert _{L^{2}(\R^N)}\leq  C(1+t)^{-N/8-k/4}\left\Vert \WF^0 \right\Vert _{L^{1}(\R^N)}+Ce^{-ct}\left\Vert \nabla^{k}\WF^0 \right\Vert _{L^{2}(\R^N)},
\end{equation*}
for any $t\geq 0$ and $0\leq k\leq s$, where $\WF(x,t)=(\tau u_{tt}+u_t,\nabla (\tau u_t+u),\theta)^T(x,t)$ and $\WF^0=\WF(x,0)$ (see Theorem \ref{Decay_Fourier_2}). On the other hand, for the Cattaneo model when $\tau=\beta$ we have:
 \begin{equation*}\label{Main_estimate_Cattaneo_law_tau=beta}
\left\Vert \nabla^{k}\WC\left( t\right) \right\Vert _{L^{2}(\R^N)}\leq  C(1+t)^{-N/8-k/4}\left\Vert \WC^0 \right\Vert _{L^{1}(\R^N)}+C(1+t)^{-\ell/3}\left\Vert \nabla^{k+\ell}\WC^0 \right\Vert _{L^{2}(\R^N)},
\end{equation*}
where $\WC(x,t)=(\tau u_{tt}+u_t,\nabla (\tau u_t+u),\theta,q)^T(x,t)$, $\WC^0=\WC(x,0)$ and $\ell$ is a nonnegative integer such that $k+\ell\leq s$ (see Theorem \ref{thm:CattaneoDecay_tau=beta}). So, roughly speaking, adding the heat conduction to the problem yields asymptotic  stability of the solution in the case $\beta=\tau$, a problem that is not asymptotically stable in the absence of heat conduction (see Sections  \ref{stabFourier_eta0} and \ref{stabFourier_3}). The heat conduction acts as a damping that suffices to obtain the asymptotic stability of the problem, but with decay rates that are slower than the ones obtained in the case $0<\tau<\beta$, which makes sense as when $\tau=\beta$ the heat conduction is our only source of dissipation. Also, and unlike the Fourier model (as before), the Cattaneo model exhibits again the regularity-loss phenomenon, with a loss of regularity apparently higher than in the $0<\tau<\beta$ case. 
In fact it is known for many mathematical models that the dissipation induced by the heat conduction is usually weaker than the frictional  dissipation or viscoelastic dissipation. For instance, similar decay rates have been obtained recently for the Timoshenko system coupled with heat conduction in  \cite{Mori_Kawashima_Fourier} and \cite{Khader_Said_Houri_Asy}.

Another result we see in the present work is the following one. In \cite{PellSaid_2016} we proved that the condition \eqref{condition} was a sufficient condition for the stability of the Cauchy problem for the standard linear model without heat conduction, but we did not prove it was a necessary one. In Section \ref{stabFourier_eta0} we  use the Routh-Hurwitz theorem to complete this stability analysis proving that this condition \eqref{condition} is also a necessary condition for the asymptotic stability of the model \eqref{Rivera_Model_Cattaneo} with $\eta=0$, this result is important since in bounded domain only numerical evidences were presented. See \cite{Conejero} and \cite{KatLasPos_2012}.  
Also the condition  $0<\tau\leq \beta$ is a sufficient (but not necessary) condition for the asymptotic stability for the standard linear model with the Fourier heat coupling (\eqref{Rivera_Model_Cattaneo} with $\eta>0$ and $\tau_0=0$), which completes the stability results given in Sections \ref{stabFourier_1} and \ref{stabFourier_3}. The same techniques do not give any information for the Cattaneo case.

Finally, to see the optimality of the previous decay and regularity results, we compute the asymptotic expansions of the real parts of the eigenvalues associated to the Fourier transform of the solution, both when the Fourier parameter $|\xi|$ tends to 0 and to $\infty$. Roughly speaking, the behavior of the real part when $|\xi|\to 0$ determines the decay rate of the solution, while the behaviour of the real parts for large frequencies is responsible for the required regularity of the initial data for the decay to be fulfilled (see Remark 4.4 in \cite{PellSaid_2016} for more details). We see that in all cases except the last one (Cattaneo when $\tau=\beta$) the decay rates are optimal since the asymptotic behaviour of the solutions (which is expected from the asymptotic expansion of the eigenvalues) agrees with the one given by the previous decay results (see Remarks \ref{rmk:expansion1}, \ref{rmk:expansion2}, \ref{rmk:expansion3} and \ref{rmk:decayCatt_tau=beta}).

To summarize, the main interest of this work relies on the following results. First, adding the heat conduction does not improve the decay in the case $0<\tau<\beta$. It may even destroy the nice property of the solution as in the Cattaneo case (regularity loss phenomenon). Second, for $\tau=\beta$, adding the heat conduction provides a stability to the solution which is (as we prove) unstable in the absence of the heat  conduction. Finally, there is a significant difference between the Fourier law and the Cattaneo one (in terms of the regularity loss). Such difference cannot be seen when $\tau=0$. This makes  the limit problem $\tau\to 0$ very interesting.

This paper is organized as follows. In Section \ref{Section_Fourier}, we investigate the Fourier-law model, while Section \ref{Section_Cattaneo} is devoted to the analysis of the Cattaneo-law  model. More concretely, in Section \ref{Well-posed-Fourier} we show the well-posedness of the Fourier model, and in Sections \ref{stabFourier_1} and \ref{stabFourier_3} we prove the stability results for the Fourier model when $\tau<\beta$ and $\tau=\beta$, respectively,  and discuss the optimality of the corresponding decay rates comparing them with the one expected from the asymptotic expansion of the eigenvalues. In Section \ref{stabFourier_eta0} we  see that, indeed, $0<\tau<\beta$ is a necessary and sufficient condition for the stability of \eqref{Rivera_Model_Cattaneo} when $\eta=0$ (no heat coupling), a result which represents a novelty for this problem. We  also use the same techniques to discuss wether this condition is also necessary when $\eta>0$. In Section \ref{Section_Cattaneo} we show the well-posedness for the Cattaneo model, and  Sections \ref{stabCattaneo_1} and \ref{stabCattaneo_2} are devoted to the stability results for the Cattaneo model when when $\tau<\beta$ and $\tau=\beta$, respectively.


\section{The Fourier-law model}\label{Section_Fourier}
In this section, we consider the standard linear model of viscoelasticity coupled with heat conduction of Fourier type. That is, we consider the  system \eqref{Rivera_Model_Cattaneo} with $\tau_0=0$:
\begin{subequations}
\begin{equation}\label{eq:Fourier}
\left\{\begin{array}{l}
\tau u_{ttt}+u_{tt}-a^{2}\Delta u-a^{2}\beta \Delta u_{t}+\eta \Delta \theta =0,\vspace{0.2cm}\\
\theta _{t}-\gamma\kappa\Delta \theta-\tau\eta\Delta  u_{tt}-\eta\Delta u_{t}=0,
\end{array}\right.
\end{equation}%
where $x\in \mathbb{R}^N,\,t\geq 0$, with the following initial data
\begin{equation}
u\left( x,0\right) =u_{0}\left( x\right) ,\quad u_{t}\left( x,0\right)
=u_{1}\left( x\right) ,\quad u_{tt}\left( x,0\right) =u_{2}\left( x\right)
,\quad \theta \left( x,0\right) =\theta _{0}\left( x\right).
\label{Initial_data}
\end{equation}
\label{Main_system}
\end{subequations}

Without loss of generality we will consider $a=1$ and $\gamma=\kappa=1$ for the rest of this section. First (in Section \ref{Well-posed-Fourier}), we are giving the appropriate functional setting to prove the well-posedness of this problem. Second (Sections \ref{stabFourier_1}, \ref{stabFourier_eta0} and \ref{stabFourier_3}) we show stability results of the solution of \eqref{Main_system}, discussing the cases $0<\tau<\beta$ and $\tau=\beta$, and  $\eta=0$ and $\eta>0$,  separately. 

\subsection{Functional setting and well-posedness of the Fourier-law model}\label{Well-posed-Fourier}

We are going to follow and adapt to the present problem the ideas for the functional setting and the well-posedness in \cite{ABFRS_2012} and in \cite{PellSaid_2016}. We are going to see that this is a well-posed problem when $0<\tau\leq \beta$.

First, we need to write problem \eqref{Main_system} as a first-order evolution equation. We take $v=u_t$ and $w=u_{tt}$ and we rewrite the previous problem as:

\begin{equation}\label{eq:ev_equation_F}
\left\{
\begin{array}{l}
\dfrac{d}{dt}U(t) = \AF U(t) ,\ \ t\in[0,+\infty)\vspace{0.2cm} \\
U(0)=U_0\end{array}
\right.
\end{equation}
where $U(t)=(u,v,w,\theta)^T$, $U_0=(u_0,u_1,u_2,\theta_0)^T$ and $\AF:\D(\AF)\subset \cH_{\mathrm{F}}\longrightarrow \cH_{\mathrm{F}}$ is the following linear operator
\begin{equation*}
\AF \left(
\begin{array}{c}
u \\
v \\
w \\
\theta
\end{array}%
\right)=
\left(
\begin{array}{c}
v \\
w \\
\frac{1}{\tau}\Delta(u+\beta v-\eta \theta)-\frac{1}{\tau}w \\
\Delta(\eta v+\tau\eta w+\theta )
\end{array}%
\right).
\end{equation*}
Following \cite{ABFRS_2012} and \cite{PellSaid_2016}, we introduce the energy space
$$\cH_{\mathrm{F}}= H^1(\R^N)\times H^1(\R^N)\times L^2(\R^N)\times L^2(\R^N)$$
with the following inner product
\begin{equation*}
\begin{array}{ll}
\langle (u,v,w,\theta),\,(u_1,v_1,w_1,\theta_1)  \rangle_{\cH_{\mathrm{F}}} = & \tau(\beta-\tau) \displaystyle\int_{\mathbb{R}^N}\nabla v\cdot\nabla \overline{v}_1  \,dx +\int_{\mathbb{R}^N} \nabla(u+\tau v)\cdot \nabla(\overline{u}_1+\tau \overline{v}_1)\,dx  \vspace{0.2cm}\\
& + \,\displaystyle\int_{\mathbb{R}^N} (v+\tau w) (\overline{v}_1+\tau \overline{w}_1) \,dx
 \, + \,\displaystyle\int_{\mathbb{R}^N} (u+\tau v) (\overline{u}_1+\tau \overline{v}_1) \,dx  \vspace{0.2cm}\\
&  + \,\displaystyle\int_{\mathbb{R}^N} v \overline{v}_1 \,dx + \,\displaystyle\int_{\mathbb{R}^N} \theta \overline{\theta}_1 \,dx
\end{array}
\end{equation*}
and the corresponding norm
\begin{equation}\label{normF}
\begin{array}{ll}
\left\| (u,v,w,\theta)\right\|^2_{\cH_{\mathrm{F}}} = & \tau(\beta-\tau) \|\nabla v\|^2_{L^2(\R^N)}  + \|\nabla(u+\tau v)\|^2_{L^2(\R^N)} \vspace{0.2cm} \\
& +  \|v+\tau w\|^2_{L^2(\R^N)} + \|u+\tau v\|^2_{L^2(\R^N)} + \|v\|^2_{L^2(\R^N)}+ \|\theta\|^2_{L^2(\R^N)}.
\end{array}
\end{equation}

\begin{remark}\label{rem:equivnorms}
It can be seen that this  new norm is equivalent to the usual one in $\cH_{\mathrm{F}}= H^1(\R^N)\times H^1(\R^N)\times L^2(\R^N)\times L^2(\R^N)$. Actually, and as we said in \cite{PellSaid_2016}, observe that this new norm is slightly different from the one introduced in \cite{ABFRS_2012}: as in unbounded domains we do not have the Poincar\'e inequality, the new terms are necessary so that the new norm \eqref{normF} is equivalent to the usual one in $\cH_{\mathrm{F}}$. 
\end{remark}

We consider \eqref{eq:ev_equation_F} in the Hilbert space $\cH_{\mathrm{F}}$ with domain
\begin{equation*}\label{eq:domain_F}
\D(\AF) =\left\{ (u,v,w,\theta)\in \cH_{\mathrm{F}}; w,\theta\in H^1(\R^N) \textrm{ and } u+\beta v-\eta\theta,\, \eta v+\tau\eta w+\theta \in H^2(\R^N) \right\}.
\end{equation*}

\begin{theorem}\label{thm:semigroup_F}
Under the dissipative condition $0<\tau\leq\beta$, the operator $\AF$  is the generator of a $C_0$-semigroup of contractions on $\cH_{\mathrm{F}}$. In particular, for any $U_0\in \D(\AF)$, there exists a unique function $U\in C^1([0,+\infty);\cH_{\mathrm{F}})\cap  C([0,+\infty);\D(\AF))$ satisfying \eqref{eq:ev_equation_F}.
\end{theorem}
\begin{pf}
As in \cite{PellSaid_2016}, we are going to consider the perturbed problem
  \begin{equation*}\label{eq:perturbed_equation}
\left\{
\begin{array}{l}
\dfrac{d}{dt}U(t) = \AbF U(t) ,\ \ t\in[0,+\infty)\vspace{0.2cm} \\
U(0)=U_0\end{array}
\right.
\end{equation*}
where
$$\AbF \left(
\begin{array}{c}
u \\
v \\
w \\
\theta
\end{array}\right) =
(\AF +B_{\mathrm{F}}) \left(
\begin{array}{c}
u \\
v \\
w \\
\theta
\end{array}\right) =
\left(
\begin{array}{c}
v \\
w \\
\frac{1}{\tau}\Delta(u+\beta v-\eta \theta)-\frac{1}{\tau}w- \frac{1}{\tau} u - v - \frac{1}{\tau^2} v \\
\Delta(\eta v+\tau\eta w +\theta)
\end{array}
\right).$$

We are going to prove that, if $0<\tau\leq \beta$, $\AbF$ is a maximal monotone operator in $\cH_{\mathrm{F}}$. Using the Lummer-Phillips theorem (see, for instance, Theorem 4.3 in Chapter 1 of \cite{Pazy}), this allows us to say that,  $\AbF$ is the generator of a $C_0$-semigroup of contractions on $\cH_{\mathrm{F}}$. As $\AbF$ is a bounded perturbation of $\AF$ in $\cH_{\mathrm{F}}$, we can say that $\AF$ is also the generator of a $C_0$-semigroup of contractions on $\cH_{\mathrm{F}}$ when $0<\tau\leq \beta$ (see, for instance, Theorem 1.1 in Chapter 3 of \cite{Pazy}).

To prove this, observe first that, following the same steps as in the proof of Proposition 2.1 in \cite{ABFRS_2012} and using our new inner product, we can see that, for any $U\in \D(\AbF)$,
$$\Re \langle \AbF U,U \rangle_{\cH_{\mathrm{F}}} = -(\beta-\tau) \displaystyle\int_{\mathbb{R}^N}|\nabla v|^2\, dx   -\frac{1}{\tau} \displaystyle\int_{\mathbb{R}^N}| v|^2\, dx - \displaystyle\int_{\mathbb{R}^N}| \nabla \theta|^2\, dx\leq 0$$
since $0<\tau\leq \beta$. Hence, the operator $\AbF$ is dissipative when $0<\tau\leq \beta$.

Now, following the same steps as in the proof of Proposition 2.2 in \cite{ABFRS_2012} and Theorem 2.2 in \cite{PellSaid_2016}, we can see that $\mathcal{R}(\lambda Id-\AbF)=\cH_{\mathrm{F}}$ for $\lambda=\dfrac{-1+\sqrt{1+4\tau}}{2\tau}>0$. That is, $\AbF$ is maximal in $\cH_{\mathrm{F}}$. As the proof of this fact is the same as the proof of the surjectivity in the Cattaneo-law model but taking $\tau_0=0$, we refer the reader to the Section \ref{Well-posed-Cattaneo} for the details. 

As we said, since $\AbF$ is maximal monotone when $0<\tau\leq \beta$, we can conclude that $\AbF$ (and, hence, $\AF$) generates a $C_0$ semigroup of contractions in $\cH_{\mathrm{F}}$. The regularity of the solution follows using \cite{Brezis} or \cite{Pazy}.
\end{pf}

\subsection{Stability results: the case $0<\tau<\beta$.}\label{stabFourier_1}
We show some stability results of a norm related with the solution of \eqref{Main_system}. We discuss the two cases separatedly: $0<\tau<\beta$ in the present section and $0<\beta=\tau$ in Section \ref{stabFourier_3}. In this section we show that the decay rate of the norm related to the solution of \eqref{Main_system} is $(1+t)^{-N/4}$.

Taking the Fourier transform of \eqref{Main_system} we obtain the following ODE initial value problem
\begin{subequations}
\begin{equation}\label{Eq_Fourier}
\left\{\begin{array}{l}
\tau \hat{u}_{ttt}+\hat{u}_{tt}+|\xi|^2 \hat{u}+\beta |\xi|^2 \hat{u}_{t}-\eta |\xi|^2 \hat{\theta} =0,\vspace{0.2cm}\\
\hat{\theta} _{t}+|\xi|^2 \hat{\theta}+\tau\eta|\xi|^2  \hat{u}_{tt}+\eta|\xi|^2 \hat{u}_{t}=0,
\end{array}\right.
\end{equation}%
with
\begin{equation}
\hat{u}\left( \xi,0\right) =\hat{u}_{0}\left( \xi\right) ,\quad \hat{u}_{t}\left( \xi,0\right)
=\hat{u}_{1}\left( \xi\right) ,\quad \hat{u}_{tt}\left( \xi,0\right) =\hat{u}_{2}\left( \xi\right)
,\quad \hat{\theta} \left( \xi,0\right) =\hat{\theta} _{0}\left( \xi\right)
\end{equation}%
\label{Main_system_Fourier}
\end{subequations}
where $\xi\in\R^N$. Recall that
\begin{equation*}\label{Change_Variables}
\hat{v}=\hat{u}_t,\qquad \hat{w}= \hat{u}_{tt}.
\end{equation*}%
Thus, system (\ref{Main_system_Fourier}) becomes%


\begin{subequations}
\begin{eqnarray}
&&\hat{u}_{t}-\hat{v}=0, \label{First_equation_system_Fourier}\\
&&\hat{v}_{t}-\hat{w}=0,
 \label{Second_equation_system_Fourier}\\
&&\tau\hat{w}_{t}+\hat{w}+|\xi|^2\hat{u}+\beta |\xi|^2\hat{v}-\eta|\xi|^2\hat{\theta}=0,  \label{Third_equation_system_Fourier} \\
&&\hat{\theta} _{t}+|\xi|^2\hat{\theta}+\eta\tau|\xi|^2 \hat{w}+\eta |\xi|^2\hat{v}=0,  \label{Fourth_equation_system_Fourier}
\end{eqnarray}%
with the initial data
\begin{eqnarray}
\hat{u}\left( \xi,0\right)  =\hat{u}_{0}\left( \xi\right),\quad \hat{v}\left( \xi,0\right) =v_{0}\left( \xi\right),\quad
\hat{w}\left( \xi,0\right)  =\hat{w}_{0}\left( \xi\right),\quad \hat{\theta}\left( \xi,0\right)=\hat{\theta}(\xi).
\label{Initial_condition_system_Fourier}
\end{eqnarray}
\label{System_New_2}
\end{subequations}

We can write system \eqref{System_New_2} as the following first order ODE system in the matrix form
\begin{equation}\label{ODE_Fourier}
  \hat{U}_t=\Psi(\xi)\hat{U},\quad \text{with}\quad  \Psi(\xi)=L+|\xi|^2 A,
\end{equation}
where $\hat{U}(\xi,t)=(\hat{u}(\xi,t),\hat{v}(\xi,t),\hat{w}(\xi,t),\hat{\theta}(\xi,t))^T$, and
$$L=\left(
  \begin{array}{cccc}
    0 & 1 & 0 & 0 \\
    0 & 0 & 1 & 0 \\
    0 & 0 & -\frac{1}{\tau} & 0 \\
    0 & 0 & 0 & 0 \\
  \end{array}
\right) ,\qquad
A=\left(
  \begin{array}{cccc}
    0 & 0 & 0 & 0 \\
    0 & 0 & 0 & 0 \\
    -\frac{1}{\tau} & -\frac{\beta}{\tau} & 0 & \frac{\eta}{\tau} \\
    0 & -\eta & -\tau\eta & -1 \\
  \end{array}
\right).$$

\subsubsection{Lyapunov functional}\label{sec:FourierFunctional}
In this section, we apply the energy method in the Fourier space in order to find the decay rate of the $L^2$-norm of the vector $\VF(x,t)=(\tau u_{tt}+u_t,\nabla (\tau u_t+u),\nabla u_t,\theta)^T(x,t)$, where $(u(x,t),\theta(x,t))$ is the solution of \eqref{Main_system}. We  follow the same ideas of  \cite[Section 3]{PellSaid_2016}.

The main result of the present section is the decay for this $L^2$-norm of $\VF$, given in the following theorem. 
\begin{theorem} \label{Decay_Second_system}
Assume that $0<\tau<\beta $. Let $\VF(x,t)=(\tau u_{tt}+u_t,\nabla (\tau u_t+u),\nabla u_t,\theta)^T(x,t)$, where $(u,\theta)$ is the solution of (\ref{Main_system}). Let $s$ be a nonnegative integer and let $\VF^{0}=\VF(x,0)\in
H^{s}\left(
\mathbb{R}^N
\right) \cap L^{1}\left(
\mathbb{R}^N
\right) .$ Then, the following decay estimate
\begin{equation}\label{Main_estimate_Theorem_1}
\left\Vert \nabla^{k}\VF\left( t\right) \right\Vert _{L^{2}(\R^N)}\leq  C(1+t)^{-N/4-k/2}\left\Vert \VF^0 \right\Vert _{L^{1}(\R^N)}+Ce^{-ct}\left\Vert \nabla^{k}\VF^0 \right\Vert _{L^{2}(\R^N)}
\end{equation}
holds, for any $t\geq 0$ and $0\leq k\leq s$, where $C$ and $c$ are two positive constants independent of $t$ and $\VF^0$.
\end{theorem}

\begin{remark}
Observe that, as we said before, this decay rate is the same one obtained in \cite{PellSaid_2016} for $V(x,t)=(\tau u_{tt}+u_t,\nabla (\tau u_t+u),\nabla u_t)^T(x,t)$, where $u(x,t)$ is the solution of the Cauchy problem for the case $\eta=0$ (no heat conduction) when $0<\tau<\beta$. Hence, the heat coupling under the Fourier law does not change the decay, as one may expect, or change the regularity of this norm related to the solution in this case. 
\end{remark}
To prove the previous theorem, we need to obtain the following pointwise estimate for the Fourier image of $\VF$.
\begin{proposition}\label{Main_Lemma}
Assume that $0<\tau<\beta$ and let  $({u},{\theta})$ be the solution of \eqref{Main_system}. Then for $\VF(x,t)=(\tau u_{tt}+u_t,\nabla (\tau u_t+u),\nabla u_t,\theta)^T(x,t)$ and for all $t\geq 0$, we have
\begin{equation}\label{U_Inequality}
|\hat{\VF}(\xi,t)|^2\leq Ce^{-c\rF(\xi)t}|\hat{\VF}(\xi,0)|^2,
\end{equation}
where $$\rF(\xi)=\dfrac{|\xi|^2}{(1+|\xi|^2)}$$
and $C$ and $c$ are two positive constants.
\end{proposition}

The proof of this proposition is  done using the following lemmas.

\begin{lemma}\label{dissipa_Energy}
Assume that $0<\tau<\beta$. Let $(\hat{u},\hat{v},\hat{w},\hat{\theta})(\xi,t)$ be the solution of (\ref{System_New_2}).
We define the energy functional associated to the system in the Fourier space (\ref{System_New_2}) as
\begin{equation}\label{Energy_functional}
\hat{\EF}(\xi,t):=\frac{1}{2}\left\{ | \hat{v}+\tau\hat{w}|^2+|\xi|^2|\hat{u}+\tau\hat{v}|^2+\tau(\beta-\tau)|\xi|^2|\hat{v}|^2+|\hat{\theta}|^2\right\}\geq 0.
\end{equation}
Then for all $t\geq 0$, there exist two positive constants $c_1$ and $c_2$ such that
\begin{equation}\label{Positivity_Energy}
 c_1|\hat{\VF}(\xi, t)|^2\leq \hat{\EF}(\xi,t)\leq c_2 |\hat{\VF}(\xi, t)|^2
\end{equation}
and
\begin{equation}\label{dE_dt_1}
\frac{d}{dt}\hat{\EF}(\xi,t)=-(\beta-\tau)|\xi|^2|\hat{v}|^2-|\xi|^2|\hat{\theta}|^2,
\end{equation}
where
\begin{equation*}
|\hat{\VF}(\xi, t)|^2=| \tau\hat{w}+\hat{v}|^2+|\xi|^2|\tau\hat{v}+\hat{u}|^2+|\xi|^2|\hat{v}|^2+|\hat{\theta}|^2.
\end{equation*}

\end{lemma}
\begin{remark}
Observe that this energy is the same one defined in \cite[Section 3]{PellSaid_2016} plus the heat therm.  So, the proof of this Proposition will follow the one of Lemma 3.2 in \cite{PellSaid_2016}.
\end{remark}

\begin{proof}
First observe that inequality (\ref{Positivity_Energy}) holds under the assumption $0<\tau<\beta$.

Let us now prove equality \eqref{dE_dt_1}. Summing up the equations \eqref{Second_equation_system_Fourier} and \eqref{Third_equation_system_Fourier}, we get
\begin{equation}\label{Sum_Equations}
(\hat{v}+\tau \hat{w})_t=-|\xi|^2\hat{u}-\beta |\xi|^2\hat{v}+\eta |\xi|^2\hat{\theta}.
\end{equation}
Multiplying \eqref{Sum_Equations} and  $\bar{\hat{v}}+\tau \bar{\hat{w}}$ and taking the real parts, we obtain,
\begin{eqnarray}\label{w_estimate}
\frac{1}{2}\frac{d}{dt} |\hat{v}+\tau \hat{w}|^2&=&-\tau |\xi|^2\func{Re}(\hat{u}\bar{\hat{w}})-\beta\tau |\xi|^2\func{Re}(\hat{v}\bar{\hat{w}})-|\xi|^2\func{Re}(\hat{u}\bar{\hat{v}})\notag\\
&&-\beta|\xi|^2|\hat{v}|^2+\eta |\xi|^2\func{Re}(\hat{\theta} \bar{\hat{v}})+\eta\tau |\xi|^2\func{Re}(\hat{\theta} \bar{\hat{w}}),
\end{eqnarray}
Next, multiplying the  equation \eqref{Second_equation_system_Fourier}  by $\tau (\beta-\tau )\bar{\hat{v}}$ and taking the real part, we get
\begin{equation}\label{v_Estimate}
\frac{1}{2}\tau (\beta-\tau )\frac{d}{dt}|\hat{v}|^2=\tau (\beta-\tau )\func{Re}(\hat{w}\bar{\hat{v}}).
\end{equation}
Now, multiplying the equation \eqref{Second_equation_system_Fourier} by $\tau $ and adding the result to the  equation \eqref{First_equation_system_Fourier}, we obtain
\begin{equation}\label{Equation_u_v}
(\hat{u}+\tau \hat{v})_t=\tau \hat{w}+\hat{v}.
\end{equation}
Multiplying \eqref{Equation_u_v} by $\bar{\hat{u}}+\tau \bar{\hat{v}}$ and taking the real parts, we get
\begin{equation}\label{u_v_derivative}
\frac{1}{2}\frac{d}{dt}|\hat{u}+\tau \hat{v}|^2=\tau \func{Re}(\hat{w}\bar{\hat{u}})+\tau ^2\func{Re}(\hat{w}\bar{\hat{v}})+\func{Re}(\hat{v}\bar{\hat{u}})+\tau |\hat{v}|^2.
\end{equation}
Finally, multiplying the equation \eqref{Fourth_equation_system_Fourier} by $\bar{\hat{\theta}}$ and taking the real part, we obtain
\begin{equation}\label{Equation_Theta_0}
\frac{1}{2}\frac{d}{dt}|\theta|^2=-|\xi|^2|\hat{\theta}|^2-\eta\tau|\xi|^2\func{Re} (\hat{w}\bar{\hat{\theta}})-\eta|\xi|^2\func{Re} (\hat{v}\bar{\hat{\theta}}).
\end{equation}

Now, computing $\eqref{w_estimate}+|\xi|^2\eqref{v_Estimate}+|\xi|^2\eqref{u_v_derivative}+\eqref{Equation_Theta_0}$, we obtain \eqref{dE_dt_1}, which finishes the proof of Lemma \ref{dissipa_Energy}.

\end{proof}

The next two lemmas follow the same ideas as in \cite{PellSaid_2016}.

\begin{lemma}\label{Lemma_F_1}
Let us define the functional $\F_1(\xi,t)$ as
\begin{equation}\label{F_1_Functional}
\F_1(\xi,t)=\func{Re}\left( (\bar{\hat{u}}+\tau \bar{\hat{v}})(\hat{v}+\tau \hat{w})\right).
\end{equation}
Then, for any $\epsilon_0>0$, we have
\begin{equation}\label{dF_1_dt}
\frac{d}{dt}\F_1(\xi,t)+(1-\epsilon_0)|\xi|^2|\hat{u}+\tau \hat{v}|^2
\leq  |\hat{v}+\tau \hat{w}|^2+C(\epsilon_0)|\xi|^2(|\hat{v}|^2+|\hat{\theta}|^2).
\end{equation}

\end{lemma}

\begin{proof}
The proof of this Lemma can be done following the proof of Lemma 3.2 of \cite{PellSaid_2016}. In order to make the paper self-contained, we recall the proof here. Multiplying equation \eqref{Sum_Equations} by $\bar{\hat{u}}+\tau \bar{\hat{v}}$ and equation \eqref{Equation_u_v} by $\bar{\hat{v}}+\tau \bar{\hat{w}}$ we get, respectively,
\begin{eqnarray*}
(\hat{v}+\tau \hat{w})_t(\bar{\hat{u}}+\tau \bar{\hat{v}}) &=&(-|\xi|^2\hat{u}-\beta |\xi|^2\hat{v}+\eta |\xi|^2\hat{\theta})(\bar{\hat{u}}+\tau \bar{\hat{v}})\\
&=&(-|\xi|^2\hat{u}-\beta |\xi|^2\hat{v}-\tau |\xi|^2\hat{v}+\tau |\xi|^2\hat{v}+\eta |\xi|^2\hat{\theta}))(\bar{\hat{u}}+\tau \bar{\hat{v}})
\end{eqnarray*}
and
\begin{equation*}
(\hat{u}+\tau \hat{v})_t(\bar{\hat{v}}+\tau \bar{\hat{w}})=(\tau \hat{w}+\hat{v})(\bar{\hat{v}}+\tau \bar{\hat{w}}).
\end{equation*}
Summing up the above two equations and taking the real part, we obtain
\begin{equation}\label{F_1_main}
\frac{d}{dt}\F_1(\xi,t)+|\xi|^2|\hat{u}+\tau \hat{v}|^2-|\hat{v}+\tau \hat{w}|^2
=|\xi|^2(\tau -\beta)\func{Re}(\hat{v}(\bar{\hat{u}}+\tau \bar{\hat{v}}))+\eta|\xi|^2\func{Re}(\hat{\theta}(\bar{\hat{u}}+\tau \bar{\hat{v}}))
\end{equation}
Applying Young's inequality for any $\epsilon_0>0$, we obtain \eqref{dF_1_dt}. This ends the proof of Lemma \ref{Lemma_F_1}.
\end{proof}

\begin{lemma}\label{Lemma_F_2}
Let us define the functional $\F_2(\xi,t)$ as
\begin{equation}\label{F_2_Functional}
\F_2(\xi,t)=-\tau \func{Re}( \bar{\hat{v}}(\hat{v}+\tau \hat{w})).
\end{equation}
Then, for any $\epsilon_1,\epsilon_2>0$, we have


\begin{equation}\label{dF_2_dt_1}
\frac{d}{dt}\F_2(\xi,t)+(1-\epsilon_1)|\hat{v}+\tau \hat{w}|^2 \leq
 C(\epsilon_1,\epsilon_2,\epsilon_3)(1+|\xi|^2)|\hat{v}|^2+\epsilon_2|\xi|^2|\hat{u}+\tau \hat{v}|^2 + \epsilon_3|\xi|^2|\hat{\theta}|^2.
\end{equation}
\end{lemma}

\begin{proof}
The proof of this Lemma can be done following the proof of Lemma 3.4 of \cite{PellSaid_2016}. As before, in order to make the paper self-contained, we include the whole proof here. Multiplying the  equation  \eqref{Second_equation_system_Fourier} by $-\tau (\bar{\hat{v}}+\tau \bar{\hat{w}})$ and \eqref{Sum_Equations} by $-\tau \bar{\hat{v}}$,
we obtain, respectively,
\begin{equation*}
-\tau \hat{v}_t(\bar{\hat{v}}+\tau \bar{\hat{w}})=-\tau \hat{w}(\bar{\hat{v}}+\tau \bar{\hat{w}})
\end{equation*}
and
\begin{eqnarray*}
-\tau (\hat{v}+\tau \hat{w})_t\bar{\hat{v}}&=&(\tau |\xi|^2\hat{u}+\beta\tau |\xi|^2\hat{v}-\eta\tau |\xi|^2\hat{\theta})\bar{\hat{v}}\\
&=& \Big(\tau |\xi|^2\hat{u}+\tau \beta |\xi|^2\hat{v}+\tau ^2|\xi|^2\hat{v}-\tau ^2|\xi|^2\hat{v}-\eta \tau|\xi|^2\hat{\theta}+(\hat{v}+\tau \hat{w})-(\hat{v}+\tau \hat{w})\Big)\bar{\hat{v}}.
\end{eqnarray*}
Summing up the above two equations and taking the real parts, we obtain
\begin{eqnarray*}
&&\frac{d}{dt}\F_2(\xi,t)+|\hat{v}+\tau \hat{w}|^2-\tau (\beta-\tau )|\xi|^2|\hat{v}|^2\notag\\
&=&\tau |\xi|^2\func{Re}\left\{(\hat{u}+\tau \hat{v})\bar{\hat{v}}\right\}+\func{Re}\left\{(\hat{v}+\tau \hat{w})\bar{\hat{v}}\right\}-\eta \tau |\xi|^2\func{Re}(\hat{\theta}\bar{\hat{v}}).
\end{eqnarray*}
Applying Young's inequality, we obtain the estimate \eqref{dF_2_dt_1} for any $\epsilon_1,\epsilon_2,\epsilon_3>0$.
\end{proof}

\begin{proof}[Proof  of Proposition \ref{Main_Lemma}]
Again, we follow the same idea that in the proof of Proposition 2.1 in \cite{PellSaid_2016}. We define the Lyapunov functional $\LF(\xi,t)$ as
\begin{equation}\label{Lyapunov}
\LF(\xi,t)=\gamma_0\hat{\EF}(\xi,t)+\frac{|\xi|^2}{1+|\xi|^2}\F_1(\xi,t)+\gamma_1\frac{|\xi|^2}{1+|\xi|^2}\F_2(\xi,t),
\end{equation}
where $\gamma_0$ and $\gamma_1$ are positive numbers that will be fixed later on.

Taking the derivative of \eqref{Lyapunov} with respect to $t$ and making use of \eqref{dE_dt_1}, \eqref{dF_1_dt} and \eqref{dF_2_dt_1}, we obtain
\begin{eqnarray*}
&&\frac{d}{dt}\LF(\xi,t)+\Big(\gamma_1(1-\epsilon_1)-1\Big)\frac{|\xi|^2}{1+|\xi|^2}|\hat{v}+\tau \hat{w}|^2\notag\\
&&+\Big((1-\epsilon_0)-\gamma_1\epsilon_2\Big)\frac{|\xi|^2}{1+|\xi|^2}(|\xi|^2|\hat{u}+\tau \hat{v}|^2)\notag\\
&&+\Big(\gamma_0(\beta-\tau )-C(\epsilon_0)-\gamma_1C(\epsilon_1,\epsilon_2,\epsilon_3)\Big)|\xi|^2|\hat{v}|^2\notag\\
&&+\Big(\gamma_0-C(\epsilon_0)-\gamma_1\epsilon_3\Big)|\xi|^2|\hat{\theta}|^2\leq 0,
\end{eqnarray*}
where we used the fact that $|\xi|^2/(1+|\xi|^2)\leq 1$.
In the above estimate, we can fix our constants in such a way that the previous coefficients are positive. This can be achieved as follows: we pick $\epsilon_0$ and $\epsilon_1$ small enough such that $\epsilon_0<1$ and $\epsilon_1<1$. After that, we take $\gamma_1$ large enough such that
\begin{equation*}
\gamma_1>\frac{1}{1-\epsilon_1}.
\end{equation*}
Once $\gamma_1$ and $\epsilon_0$ are fixed, we select $\epsilon_2$ small enough such that
\begin{equation*}
\epsilon_2<\frac{1-\epsilon_0}{\gamma_1}.
\end{equation*}
Finally, and recalling that $\tau<\beta$, we may choose $\gamma_0$ large enough such that
\begin{equation*}
\gamma_0>\max\left\{\frac{C(\epsilon_0)+\gamma_1C(\epsilon_1,\epsilon_2)}{\beta-\tau }, C(\epsilon_0)+\gamma_1\epsilon_3\right\}
\end{equation*}
Consequently, we deduce that there exists a positive constant $\gamma_2$ such that for all $t\geq 0$,
\begin{eqnarray}\label{dL_dt_2}
&&\frac{d}{dt}\LF(\xi,t)+\gamma_2\frac{|\xi|^2}{1+|\xi|^2}\hat{\EF}(\xi,t)\leq 0.
\end{eqnarray}

On the other hand, it is not difficult to see that from \eqref{Energy_functional}, \eqref{F_1_Functional}, \eqref{F_2_Functional} and \eqref{Lyapunov} and for $\gamma_0$, large enough, that there exists two positive constants $\gamma_3$ and $\gamma_4$ such that
\begin{equation}\label{Equival_E_L}
\gamma_3\hat{\EF}(\xi,t)\leq \LF(\xi,t)\leq \gamma_4\hat{\EF}(\xi,t).
\end{equation}
Combining \eqref{dL_dt_2} and \eqref{Equival_E_L}, we deduce that there exists a positive constant $\gamma_5$ such that for all $t\geq 0$,
\begin{equation*}\label{Estimate_L_main}
\frac{d}{dt}\LF(\xi,t)+\gamma_5\frac{|\xi|^2}{1+|\xi|^2}\LF(\xi,t)\leq 0.
\end{equation*}
A simple application of Gronwall's lemma leads to the estimate \eqref{U_Inequality}, as $\LF(\xi,t)$ and the norm of $|\hat{\VF}(\xi,t)|^2$ are equivalent.
\end{proof}
Now we have all the ingredients, we can proceed with the proof of Theorem \ref{Decay_Second_system}.
\begin{proof}[Proof of Theorem \ref{Decay_Second_system}]
Using Proposition \ref{Main_Lemma}, the proof of  Theorem \ref{Decay_Second_system} is the same as the one of Theorem 3.6 in \cite{PellSaid_2016}, but we include it here for the sake of self-containedness of the present paper.

To prove Theorem \ref{Decay_Second_system}, we have by Plancherel  theorem and the estimate (\ref{U_Inequality}) that
\begin{equation}\label{Pranch_Identity}
\left\Vert \nabla^{k}\VF(t)\right\Vert _{L^{2}(\R^N)}^{2}=\int_{\R^N
}\left\vert \xi \right\vert ^{2k}\vert \hat{\VF}(\xi ,t)\vert^{2}d\xi
\leq C\int_{\R^N
}\left\vert \xi \right\vert ^{2k}e^{-c\rF(\xi)t}\vert \hat{\VF}(\xi,0)\vert^{2}d\xi,
\end{equation}
It is obvious that the term on the right-hand side of (\ref{Pranch_Identity}) depends on the behavior of the function $\rF(\xi)$. Since
\begin{equation*}\label{rho_behavior}
\rF(\xi)\geq\left\{
\begin{array}{ll}
\frac{1}{2}|\xi|^2,& \text{for } |\xi|\leq 1,\vspace{0.2cm} \\
\frac{1}{2}, & \text{for } |\xi|\geq 1,
\end{array}%
\right.
\end{equation*}%
then it is natural to write the integral on the right-hand side of (\ref{Pranch_Identity}) as
\begin{eqnarray}
\int_{\mathbb{R}
}\left\vert \xi \right\vert ^{2k}e^{-c\rF(\xi)t}\vert \hat{\VF}(\xi,0)\vert^{2}d\xi &=
&\int_{\left\vert \xi \right\vert \leq 1}\left\vert \xi \right\vert
^{2k}e^{-c\rF (\xi )t}\vert \hat{\VF}(\xi ,0)\vert ^{2}d\xi+\int_{\left\vert \xi \right\vert \geq 1}C\left\vert \xi \right\vert
^{2k}e^{-c\rF(\xi )t}\vert \hat{\VF}(\xi ,0)\vert ^{2}d\xi
\label{inequality L1+L2}\notag \\
&:= &L_{1}+L_{2}.  \label{L_1_L_2_estimate}
\end{eqnarray}
Concerning the integral $L_1$, we have
\begin{equation}
L_{1}
\leq C\sup_{\vert \xi \vert \leq 1} \vert \hat{\VF}
(\xi ,0)\vert ^{2} \int_{\left\vert \xi \right\vert \leq
1}\left\vert \xi \right\vert ^{2k}e^{-\frac{c}{2}|\xi| {^2}t}
d\xi \leq C\Vert \hat{\VF^0}(t)\Vert _{L^{\infty
}}^{2}\int_{\left\vert \xi \right\vert \leq 1}\left\vert \xi \right\vert
^{2k}e^{-\frac{c}{2}|\xi|
{^2}%
t} d\xi . \nonumber
\end{equation}
Passing to the polar coordinates and using the following inequality (see Lemma 3.5 in \cite{PellSaid_2016}):
\begin{equation*}\label{Inequality_exponential}
\int_{\vert \xi \vert \leq 1}\left\vert \xi \right\vert ^{j }e^{-c|\xi| ^{2}t}d\xi \leq
C\left( 1+t\right) ^{-\left( j +N\right) /2},\ \ \textrm{for}\ N\geq 1
\end{equation*}
we deduce that
\begin{eqnarray}\label{L_1_inequality}
L_{1} \leq  C(1+t)^{-N/2-k}\left\Vert \VF^0\right\Vert _{L^{1
}(\R^N)}^{2}.
\end{eqnarray}
On the other hand, we have
\begin{equation}\label{L_2_estimate}
L_{2}
\leq Ce^{-\frac{1}{2}t} \int_{\left\vert \xi \right\vert \geq 1}\left\vert \xi
\right\vert ^{2k}\left\vert \hat{\VF}(\xi ,0)\right\vert
^{2}d\xi \leq Ce^{-ct}\left\Vert \nabla ^{k}\VF^0\right\Vert _{L^{2}(\R^N)}^{2}
\end{equation}
for a certain $c>0$. Consequently, the estimate (\ref{Main_estimate_Theorem_1}) follows by combining (\ref{L_1_L_2_estimate}), (\ref{L_1_inequality}) and (\ref{L_2_estimate}). Thus the proof of Theorem \ref{Decay_Second_system} is finished.
\end{proof}

\subsubsection{Asymptotic behaviour of the eigenvalues}\label{sec:eigFourier1}
In this section, we use the eigenvalues expansion method to show that the asymptotic expansion of the eigenvalues near zero and near infinity agrees with the decay rate obtained in Theorem \ref{Decay_Second_system}.
\begin{remark}\label{rmk:expansion1}
The decay rate in Theorem \ref{Decay_Second_system} comes from the exponent
$$\rF(\xi)= \dfrac{|\xi|^2}{1+|\xi|^2}$$
of Proposition \ref{Main_Lemma}. Observe that $\rF(\xi)\sim |\xi|^2$ when $|\xi|\to 0$, and $\rF(\xi)\sim 1$ when $|\xi|\to \infty$. This is the asymptotic behaviour we expect (and, actually, will obtain) for the real parts of the slowest  eigenvalues in these cases (see Lemmas \ref{lemma:eig0_Fourier1} and \ref{lemma:eiginfty_Fourier1} below).
\end{remark}

\begin{lemma}\label{lemma:eig0_Fourier1}
Assume that $0<\tau<\beta$.  Then the  real parts of the eigenvalues of the matrix $\Psi(\xi)$ defined in \eqref{ODE_Fourier} have the following expansion as $|\xi|\rightarrow 0$:
\begin{equation}\label{Eg_Exp_Four_0}
\left\{
\begin{array}{ll}
\func{Re}(\lambda_{1,2})(\xi) &= -\dfrac{\beta-\tau}{2}|\xi|^2+O(|\xi|^3),\vspace{0.2cm} \\
  \func{Re}(\lambda_{3})(\xi) &= -|\xi|^2+O(|\xi|^3), \vspace{0.2cm}\\
 \func{Re}(\lambda_{4})(\xi) &= -\dfrac{1}{\tau}+O(|\xi|^2).
 \end{array}
\right.
\end{equation}
\end{lemma}
\begin{proof}

The characteristic polynomial of \eqref{ODE_Fourier} can be computed in terms of $\zeta=i|\xi|$ as
\begin{eqnarray}\label{Chara_Pol_Fourier}
p_{\mathrm{F}}(\lambda)=\det( L-\zeta^2A-\lambda Id),
\end{eqnarray}
where
\begin{eqnarray}\label{charpol_Fourier}
&& \tau\det( L-\zeta^2A-\lambda Id)=\\
\nonumber && \tau\lambda^4+\left(1-\tau\zeta^2\right)\lambda^3+\left( \tau\eta^2\zeta^2-\beta-1\right)\zeta^2\lambda^2 +
\left( (\eta^2+\beta)\zeta^2-1\right)\zeta^2\lambda +{\zeta^4}
\end{eqnarray}
Let us denote by $\lambda_j(\zeta)$, $j=1,\ldots,4$, the roots of \eqref{charpol_Fourier}. We can compute their asymptotic expansions $$\lambda_j(\zeta) =\lambda_j^0+\lambda_j^1\zeta+\lambda_j^2\zeta^2+\ldots,\ j=1,\ldots,4,\ \textrm{ for } \zeta\to 0$$
as
\begin{eqnarray*}
  \lambda_{1,2}(\zeta) &=& \pm \zeta -\frac{\tau-\beta}{2}\zeta^2+O(\zeta^3) \\
  \lambda_{3}(\zeta) &=& \zeta^2+O(\zeta^3) \\
  \lambda_{4}(\zeta) &=& -\frac{1}{\tau}+O(\zeta^2).
\end{eqnarray*}
Consequently, as $\zeta=i|\xi|$, we obtain  the  asymptotic behavior \eqref{Eg_Exp_Four_0} for the real parts of the previous eigenvalues when $|\xi|\to 0$.
Observe that when $0<\tau<\beta$ (dissipative case for the standard linear model) all the  real parts  in \eqref{Eg_Exp_Four_0} are negative.
\end{proof}
\begin{lemma}\label{lemma:eiginfty_Fourier1}
Assume that $0<\tau<\beta$.  Then the  real parts of the eigenvalues of the matrix $\Psi(\xi)$ defined in \eqref{ODE_Fourier} has the following expansion as $|\xi|\rightarrow \infty$:
\begin{equation}\label{Eg_Exp_Four_infty}
\left\{
\begin{array}{ll}
 \func{Re}(\lambda_{1})(\xi) &= -\dfrac{\beta+\eta^2-\sqrt{(\beta+\eta^2)^2-4\tau\eta^2}}{2\tau\eta^2}+O(|\xi|^{-1}) \vspace{0.2cm}\\
  \func{Re}(\lambda_{2})(\xi) &= -\dfrac{\beta+\eta^2+\sqrt{(\beta+\eta^2)^2-4\tau\eta^2}}{2\tau\eta^2}+O(|\xi|^{-1}) \vspace{0.2cm}\\
  \func{Re}(\lambda_{3})(\xi) &= -\dfrac{1-\sqrt{1-4\eta^2}}{2}|\xi|^2+O(|\xi|) \vspace{0.2cm}\\
  \func{Re}(\lambda_{4})(\xi) &= -\dfrac{1+\sqrt{1-4\eta^2}}{2}|\xi|^2+O(|\xi|).
 \end{array}
\right.
\end{equation}
In particular, all of them are negative.
\end{lemma}
\begin{proof}
We now proceed with the asymptotic behaviour of the eigenvalues when $|\xi|\to\infty$. Following Kawashima et al. (\cite{IHK08}) and taking $\nu=\zeta^{-1}=(i|\xi|)^{-1}$, we write equation \eqref{ODE_Fourier} as
\begin{equation}\label{ODE_Fourier_infty}
  \hat{U}_t= \nu^{-2}\left(L\nu^{2}- A\right)U.
\end{equation}
We can now compute $\det(L\nu^{2}- A-\mu Id)$, the characteristic polynomial of equation \eqref{ODE_Fourier_infty}, as:
\begin{eqnarray}\label{charpol_Fourier_infty}
&& \tau\det(L\nu^{2}- A-\mu Id)= \\
\nonumber && \tau\mu^4 + \left({\nu^2}-\tau\right)\, \mu^3 + \left(\eta^2\tau-(\beta+1)\nu^2\right) \mu^2 + (\eta^2+\beta-\nu^2)\nu^2 \mu + \nu^4
\end{eqnarray}
Observe that if $\lambda(\xi)$ is an eigenvalue \eqref{ODE_Fourier}, then $$\mu(\nu)=\nu^2\lambda=\zeta^{-2}\lambda=-|\xi|^{-2}\lambda$$
is an eigenvalue of \eqref{ODE_Fourier_infty}.

We can now compute the asymptotic expansion
$$\mu_j(\nu) =\mu_j^0+\mu_j^1\nu+\mu_j^2\nu^2+\mu_j^3\nu^3+\mu_j^4\nu^4+\ldots,\  j=1,\ldots,4$$
of the roots of the characteristic polynomial \eqref{charpol_Fourier_infty} when $\nu\to 0$ (that is, when $|\xi|\to\infty$), to get

\begin{eqnarray*}
  \mu_{1}(\nu) &=& -\frac{\beta-\sqrt{(\beta+\eta^2)^2-4\tau\eta^2}+\eta^2}{2\tau\eta^2}\nu^2+O(\nu^3), \\
  \mu_{2}(\nu) &=& -\frac{\beta+\sqrt{(\beta+\eta^2)^2-4\tau\eta^2}+\eta^2}{2\tau\eta^2}\nu^2+O(\nu^3), \\
  \mu_{3}(\nu) &=& \frac{1-\sqrt{1-4\eta^2}}{2}+O(\nu), \\
  \mu_{4}(\nu) &=& \frac{1+\sqrt{1-4\eta^2}}{2}+O(\nu).
\end{eqnarray*}
Consequently, as $$\lambda_j(\xi) =-\mu_j^0|\xi|^2+\mu_j^1 i|\xi|+\mu_j^2-\mu_j^3 i|\xi|^{-1}-\mu_j^4 |\xi|^{-2}+O(|\xi|^{-3}),\quad j=1,\ldots,4,$$ we have the asymptotic behavior \eqref{Eg_Exp_Four_infty} for the real parts of the previous eigenvalues when $|\xi|\to \infty$.
\begin{eqnarray*}
  \func{Re}(\lambda_{1})(\xi) &=& -\frac{\beta+\eta^2-\sqrt{(\beta+\eta^2)^2-4\tau\eta^2}}{2\tau\eta^2}+O(|\xi|^{-1}), \\
  \func{Re}(\lambda_{2})(\xi) &=& -\frac{\beta+\eta^2+\sqrt{(\beta+\eta^2)^2-4\tau\eta^2}}{2\tau\eta^2}+O(|\xi|^{-1}), \\
  \func{Re}(\lambda_{3})(\xi) &=& -\frac{1-\sqrt{1-4\eta^2}}{2}|\xi|^2+O(|\xi|), \\
  \func{Re}(\lambda_{4})(\xi) &=& -\frac{1+\sqrt{1-4\eta^2}}{2}|\xi|^2+O(|\xi|).
\end{eqnarray*}
Observe that in all cases we have $\func{Re}(\lambda_{j})(\xi)<0$, $j=1,\ldots,4$.

\end{proof}

\subsection{The necessary conditions for asymptotic stability for $\eta\geq 0$.}\label{stabFourier_eta0}
First, assume that $\eta=0$ in \eqref{Rivera_Model_Cattaneo}. That is, we want to deal with the Cauchy problem for the standard linear solid model without any heat coupling (also called Moore-Gibson-Thompson equation). As we have explained in Section \ref{Introduction}, this model was studied by the same authors of the present paper in \cite{PellSaid_2016} where, in particular, the condition \eqref{condition} was only proved to be a sufficient condition for the asymptotic stability of that problem. The same had been proven for the standard linear problem in a bounded domain (see \cite{KatLasPos_2012}, for instance). In this section our goal is to prove that, indeed, $0<\tau<\beta$ is also a necessary condition for the asymptotic stability of the problem when $\eta=0$, which represents a novelty for this problem as we explained in the introduction.  This result is also important in comparison with the corresponding ones for the model with heat conduction. Actually, and using the same techniques, we see that  $0<\tau\leq \beta$ is a  sufficient (but not necessary) condition for the asymptotic stability for the model with heat conduction under the Fourier law.  For the Cattaneo case, the same techniques do not add information on this, but we see in Sections \ref{stabCattaneo_1} and \ref{stabCattaneo_2} that $\tau\leq \beta$  is  sufficient condition for the stability of the Cattaneo problem. 

First, we want to give conditions for the real parts of the eigenvalues of \eqref{Rivera_Model_Cattaneo} when $\eta=0$ to be negative. Hence, we first write the characteristic polynomial of the matrix $\Phi=(L+\xi^2 A)$ in \eqref{ODE_Fourier} when $\eta=0$, that is
\begin{eqnarray}\label{Characteristic_2}
p_0(\lambda)&=&
a_0\lambda^4+a_1\lambda^3+a_2\lambda^2+a_3\lambda+a_4
\end{eqnarray}
where
\begin{eqnarray*}
&&a_0=\tau,\quad a_1=1+\tau|\xi|^2,\quad a_2=\left( \beta+1\right)|\xi|^2,\quad a_3=\left( \beta|\xi|^2+1\right)|\xi|^2,\quad a_4=|\xi|^4
\end{eqnarray*}
(see also \eqref{charpol_Fourier} for $\eta=0$). Now, as all $a_i>0$ for $i=0,\ldots, 4$, we apply the Routh--Hurwitz theorem (see \cite[p. 459]{Lavrentiev_1972}) that ensures that all
the roots of the polynomial $p_0(\lambda)$ have negative real part if and only if all the leading  minors of  the matrix
\begin{equation*}
\left(
\begin{array}{cccc}
 a_1 & a_0 & 0 & 0 \\
 a_3 & a_2 & a_1 & a_0 \\
 0 & a_4 & a_3 & a_2 \\
 0 & 0 & 0 & a_4 \\
\end{array}
\right)
\end{equation*}
are strictly positive. That is, the following determinants should be strictly positive:
\begin{eqnarray*}
&&A_1=a_1,\quad A_{2}=\det \left(
\begin{array}{cc}
a_{1} & a_{0} \\
a_{3} & a_{2}%
\end{array}%
\right),\quad A_{3} =\det \left(
\begin{array}{ccc}
a_{1} & a_{0} & 0 \\
a_{3} & a_{2} & a_{1} \\
0 & a_{4} & a_{3}%
\end{array}%
\right),\quad A_{4}=\det \left(
\begin{array}{cccc}
a_{1} & a_{0} & 0 & 0 \\
a_{3} & a_{2} & a_{1} & a_{0} \\
0 & a_{4} & a_{3} & a_2 \\
0 & 0 & 0 & a_{4}%
\end{array}%
\right)
\end{eqnarray*}
In our case, we have
\begin{eqnarray*}
&&A_1=1+\tau|\xi|^2\\
&&A_2=|\xi |^2 \left( |\xi |^2 \tau +\beta-\tau +1\right)\\
&&A_3= |\xi|^4(\beta-\tau)\left( \tau |\xi|^4+(\beta +1)|\xi|^2+1\right) \\
&&A_4=|\xi|^4 A_3.
\end{eqnarray*}
Hence, the condition $0<\tau<\beta$ is a necessary and sufficient condition to have $A_3>0$. According to the Routh-Hurwitz theorem mentioned above, this means that all the eigenvalues have strictly negative real part and, therefore, we conclude that the problem is stable if and only if $0<\tau<\beta$.

\begin{remark}
Observe that if we write \eqref{Characteristic_2} for $\beta=\tau$ we have the following characteristic polynomial
\begin{eqnarray*}
\left(\lambda +|\xi| ^2\right) \left(\lambda ^2+|\xi| ^2\right) (\lambda  \tau +1).
\end{eqnarray*}
In this case, the eigenvalues of the matrix $\Phi$ are
\begin{eqnarray*}
\lambda_1=-|\xi|^2,\quad \lambda_2=-\frac{1}{\tau},\quad \text{and}\quad \lambda_{3,4}=\pm i|\xi|.
\end{eqnarray*}
Hence, as $\func{Re}(\lambda_{3,4})=0$, we do not have asymptotic stability in this case, which is in agreement with the above result.
\end{remark}
We can apply the same technique to see what happens for the eigenvalues of the standard linear solid model with Fourier heat conduction (that is, \eqref{Rivera_Model_Cattaneo} with $\eta>0$ and $\tau_0=0$). From \eqref{charpol_Fourier} we have that the characteristic polynomial of this problem is:
\begin{eqnarray*}
p_{\mathrm{F}}(\lambda)&=&
a_0\lambda^4+a_1\lambda^3+a_2\lambda^2+a_3\lambda+a_4
\end{eqnarray*}
where now
\begin{eqnarray*}
&&a_0=\tau,\quad a_1=1+\tau|\xi|^2,\quad a_2=\left(\tau \eta^2|\xi|^2+ \beta+1\right)|\xi|^2,\\ &&a_3=\left((\eta^2+ \beta)|\xi|^2+1\right)|\xi|^2,\quad a_4=|\xi|^4.
\end{eqnarray*}
As $a_i>0$ for all $i=0,\ldots,4$, we can apply the Routh-Hurwitz Theorem again. The $A_i$ determinants for the case $\eta>0$ are
\begin{eqnarray*}
&&A_1=1+\tau |\xi|^2\\
&&A_2=|\xi |^2 \left( \eta^2\tau^2|\xi|^4+|\xi |^2 \tau +\beta-\tau +1\right)\\
&&A_{3} =\Big[ (\eta ^{2}+\beta )\eta ^{2}\tau ^{2}|\xi |^{6}+(\eta
^{2}\tau +\eta ^{2}+\beta -\tau )\tau |\xi |^{4}\Big.  \\
&&\Big. \hspace{0.8cm}+\left( (\beta -\tau )(1+\beta +\eta ^{2})+\eta ^{2}\right) |\xi
|^{2}+(\beta -\tau )\Big] |\xi |^{4} \\
&&A_4=|\xi|^4 A_3.
\end{eqnarray*}
Observe that, now, $0<\tau\leq \beta$ is a sufficient condition to have $A_1,A_2,A_3,A_4>0$ and, hence, all the eigenvalues with strictly negative real part, but it is not a necessary one. Therefore, $0<\tau\leq \beta$ is just a sufficient condition for the asymptotic stability of the standard linear problem with Fourier heat conduction. This completes the stability results given in Sections \ref{stabFourier_1} and \ref{stabFourier_3}. For the Cattaneo model, the Routh-Hurwitz criterion does not give us any information.

\subsection{Stability results: the case $\tau=\beta$} \label{stabFourier_3}
We investigate now the case $\tau=\beta$. As we have seen in the previous Section \ref{stabFourier_eta0}, in the absence of the heat conduction the condition $0<\tau<\beta$ is a necessary and sufficient condition for stability. As we have also seen in Section \ref{stabFourier_eta0}, in the presence of the heat conduction the stability happens both when $0<\tau<\beta$ and $\tau=\beta$, although these are only sufficient conditions.  In the present section we give the decay rate of a norm related to the solution when $\tau=\beta$, that agrees with the fact of having asymptotic stability also when $\tau=\beta$. This is important because, a priori, this case seems more delicate compared to the case $0<\tau<\beta$ since the only damping for the case $\tau=\beta$ is coming from the heat conduction. The consequence of this fact is still a decay but with a slower decay rate. In fact, we show that an $L^2$-norm related to the solution decays as $(1+t)^{-N/8}$, instead of the decay $(1+t)^{-N/4}$ for the case $0<\tau<\beta$ (see Theorem \ref{Decay_Second_system} and Theorem \ref{Decay_Fourier_2} below).

\subsubsection{Lyapunov functional}
In this section, we apply the energy method in the Fourier space to obtain the decay rate of the $L^2$-norm of $\WF(x,t)=(\tau u_{tt}+u_t,\nabla (\tau u_t+u), \theta)^T(x,t)$, where $(u,\theta)$ is the solution of (\ref{Main_system}) when $\tau=\beta$. As we said above, this is the main result of this section and it is summarized in the following theorem.
\begin{theorem}\label{Decay_Fourier_2}
Assume that $\beta=\tau $. Let $\WF(x,t)=(\tau u_{tt}+u_t,\nabla (\tau u_t+u), \theta)^T(x,t)$, where $(u,\theta)$ is the solution of (\ref{Main_system}). Let $s$ be a nonnegative integer and let $\WF^{0}=\WF(x,0)\in
H^{s}\left(
\mathbb{R}^N
\right) \cap L^{1}\left(
\mathbb{R}^N
\right) .$ Then, the following decay estimate
\begin{equation*}\label{Main_estimate_Theorem_2}
\left\Vert \nabla^{k}\WF\left( t\right) \right\Vert _{L^{2}(\R^N)}\leq  C(1+t)^{-N/8-k/4}\left\Vert \WF^0 \right\Vert _{L^{1}(\R^N)}+Ce^{-ct}\left\Vert \nabla^{k}\WF^0 \right\Vert _{L^{2}(\R^N)},
\end{equation*}
holds,
for any $t\geq 0$ and $0\leq k\leq s$, where $C$ and $c$ are two positive constants independent of $t$ and $\WF^0$.
\end{theorem}
The proof of this theorem is a combination of Proposition \ref{Prop_beta_tau_2} (that we state and prove below) and the proof of Theorem \ref{Decay_Second_system}.
Hence, we will omit the details of the proof of Theorem \ref{Decay_Fourier_2}.

\begin{proposition}\label{Prop_beta_tau_2}
Assume that $\tau=\beta$ and let  $({u},{\theta})$ be the solution of \eqref{Main_system}. Then for $\WF(x,t)=(\tau u_{tt}+u_t,\nabla (\tau u_t+u)
,\theta)^T(x,t)$ and for all $t\geq 0$, we have
\begin{equation}\label{U_Inequality_2}
|\hat{\WF}(\xi,t)|^2\leq Ce^{-c\vrF(\xi)t}|\hat{\WF}(\xi,0)|^2,
\end{equation}
where $$\vrF(\xi)=\frac{|\xi|^4}{(1+|\xi|^2+|\xi|^4)}$$ and $C$ and $c$ are two positive constants.
\end{proposition}

\begin{proof}[Proof of Proposition \ref{Prop_beta_tau_2}]

In order to prove Proposition \ref{Prop_beta_tau_2}, we need to  find the appropriate Lyapunov functional under the assumption $\beta=\tau$. In this case \eqref{dE_dt_1}
becomes
\begin{equation}\label{dE_dt_1_2}
\frac{d}{dt}\hat{\EcF}(\xi,t)=-|\xi|^2|\hat{\theta}|^2,
\end{equation}
where
\begin{eqnarray*}\label{Energy_Four_2}
\hat{\EcF}(\xi,t):=\frac{1}{2}\left\{ | \hat{v}+\tau\hat{w}|^2+|\xi|^2|\hat{u}+\tau\hat{v}|^2
+|\hat{\theta}|^2\right\}. 
\end{eqnarray*}

First, for $\tau=\beta$, the first term of the right-hand-side in \eqref{F_1_main} vanishes and we have (instead of \eqref{dF_1_dt}) the estimate
\begin{eqnarray}\label{dF_1_dt_tau_2}
\frac{d}{dt}\F_1(\xi,t)+(1-\epsilon_0)|\xi|^2|\hat{u}+\tau \hat{v}|^2
\leq  |\hat{v}+\tau \hat{w}|^2+C(\epsilon_0)|\xi|^2|\hat{\theta}|^2.
\end{eqnarray}
Now to get a dissipative term for $|\hat{v}+\tau\hat{w}|$ we rewrite system \eqref{Eq_Fourier} as (since $\beta=\tau$)
\begin{equation}\label{Eq_Fourier_2}
\left\{\begin{array}{l}
(\hat{v}+\tau\hat{w})_t+|\xi|^2 (\hat{u}+\tau  \hat{v})-\eta |\xi|^2 \hat{\theta} =0,\vspace{0.2cm}\\
\hat{\theta} _{t}+|\xi|^2 \hat{\theta}+\eta|\xi|^2  (\hat{v}+\tau\hat{w})=0.
\end{array}\right.
\end{equation}
We multiply the first equation in \eqref{Eq_Fourier_2} by $\bar{\hat{\theta}}$ and the second equation by $(\bar{\hat{v}}+\tau\bar{\hat{w}})$, adding the results and taking the real part, we obtain
\begin{equation*}
\frac{d}{dt}\F_3(\xi,t)+\eta|\xi|^2|\hat{v}+\tau\hat{w}|^2 =
-|\xi|^2\func{Re}( \hat{\theta}(\bar{\hat{v}}+\tau\bar{\hat{w}}))+\eta|\xi|^2|\hat{\theta}|^2-|\xi|^2\func{Re}(\bar{\hat{\theta}}(\hat{u}+\tau\hat{v})),
\end{equation*}
where
\begin{eqnarray}\label{F_3_Equation}
\F_3(\xi,t)=\func{Re}(\bar{\hat{\theta}}(\hat{v}+\tau\hat{w})).
\end{eqnarray}
Applying Young's inequality, we obtain for any $\tilde{\epsilon}_1>0$,
 \begin{equation}\label{F_3_Estimate}
\frac{d}{dt}\F_3(\xi,t)+(\eta-\tilde{\epsilon}_1)|\xi|^2|\hat{v}+\tau\hat{w}|^2
\leq (\tilde{\epsilon}_2|\xi|^4|\hat{u}+\tau \hat{v}|^2+C(\tilde{\epsilon}_1,\tilde{\epsilon}_2)(1+|\xi|^2)|\hat{\theta}|^2)
\end{equation}
Now, we define the new Lyapunov functional as follows:
\begin{equation}\label{Lyapunov_Main}
\cLF(\xi,t)=\tilde{\gamma_0}(1+|\xi|^2+|\xi|^4)\hat{\mathscr{E}}_{\mathrm{F}}(\xi,t)+|\xi|^4 \F_1(\xi,t)+\tilde{\gamma}_1|\xi|^2\F_3(\xi,t).
\end{equation}
Now, taking the derivative of \eqref{Lyapunov_Main} and using \eqref{dE_dt_1_2},  \eqref{dF_1_dt_tau_2} and  \eqref{F_3_Estimate} we obtain
\begin{eqnarray*}
&&\frac{d}{dt}\cLF(\xi ,t)+\left( (1-\epsilon _{0})-\tilde{\gamma}%
_{1}\tilde{\epsilon} _{2}\right) |\xi |^{6}|\hat{u}+\tau \hat{v}|^{2} \\
&&+\left( \tilde{\gamma}_{1}(\eta -\tilde{\epsilon} _{1})-1\right) |\xi |^{4}|%
\hat{v}+\tau \hat{w}|^{2} \\
&&+\left( \tilde{\gamma}_{0}-C(\epsilon _{0},\tilde{\epsilon} _{1},\tilde{\epsilon} _{2},\tilde{\gamma}%
_{1})\right) |\xi |^{2}(1+|\xi |^{2}+|\xi |^{4})|\hat{\theta}|^{2}
\leq 0,\qquad \forall t\geq 0.
\end{eqnarray*}
Now, we fix the above constants as follows: we select $\epsilon_0$ and $\tilde{\epsilon}_1$  such that
\begin{eqnarray*}
\epsilon_0<1,\qquad \text{and}\qquad \tilde{\epsilon}_1<\eta.
\end{eqnarray*}
Then, we take $\tilde{\gamma}_1$ large enough such that
\begin{eqnarray*}
\tilde{\gamma}_1>\frac{1}{\eta-\tilde{\epsilon}_1}
\end{eqnarray*}
and $\tilde{\epsilon}_2$  small enough such that
\begin{eqnarray*}
\tilde{\epsilon}_2<\frac{1-\epsilon_0}{\tilde{\gamma}_1}
\end{eqnarray*}
After that, we take $\tilde{\gamma}_0$ large enough  such that
 \begin{eqnarray*}
\tilde{\gamma}_{0}>C(\epsilon _{0},\tilde{\epsilon} _{1},\tilde{\epsilon} _{2},\tilde{\gamma}%
_{1})
\end{eqnarray*}
Consequently, we deduce from above that there exists a $\Lambda_0>0$ such that
\begin{eqnarray}\label{E_L_tilde_Ineq}
\frac{d}{dt}\cLF(\xi ,t)+\Lambda_0|\xi|^4\hat{\EcF}(\xi,t)\leq 0, \qquad \forall t\geq 0.
\end{eqnarray}
 On the other hand, for $\tilde{\gamma}_0$ large enough, we can find two positive constants $\Lambda_1$ and $\Lambda_2$ such that
 \begin{eqnarray}\label{E_L_tilde_Ineq_2}
\Lambda_1(1+|\xi|^2+|\xi|^4)\hat{\EcF}(\xi,t)\leq \cLF(\xi ,t)\leq \Lambda_2(1+|\xi|^2+|\xi|^4)\hat{\EcF}(\xi,t),\qquad \forall t\geq 0.
\end{eqnarray}
Combining \eqref{E_L_tilde_Ineq} and \eqref{E_L_tilde_Ineq_2}, we obtain
\begin{eqnarray}\label{E_L_tilde_Ineq_3}
\frac{d}{dt}\cLF(\xi ,t)+\Lambda_3\frac{|\xi|^4}{1+|\xi|^2+|\xi|^4}\cLF(\xi,t)\leq 0, \qquad \forall t\geq 0,
\end{eqnarray}
for some $\Lambda_3>0$.
By \eqref{E_L_tilde_Ineq_3} and \eqref{E_L_tilde_Ineq_2}, and using Gronwall's inequality, this leads to \eqref{U_Inequality_2}. This finishes the proof of Proposition \ref{Prop_beta_tau_2}.
\end{proof}

\subsubsection{Eigenvalues expansion}

In this section, we show that the asymptotic expansion of the  eigenvalues near zero and near infinity agrees with the decay rate obtained in Theorem \ref{Decay_Fourier_2}.

\begin{remark}\label{rmk:expansion2}
The decay rate in Theorem \ref{Decay_Fourier_2} comes from the exponent
$$\vrF(\xi)= \dfrac{|\xi|^4}{1+|\xi|^2+|\xi|^4}$$
of Proposition \ref{Prop_beta_tau_2}. Observe that $\vrF(\xi)\sim |\xi|^4$ when $|\xi|\to 0$, and $\vrF(\xi)\sim 1$ when $|\xi|\to \infty$. This is the asymptotic behaviour we expect for the real parts of the slower eigenvalues in these cases (see Lemmas \ref{lemma:eig0_Fourier2} and \ref{lemma:eiginfty_Fourier2} below).
\end{remark}

\begin{lemma}\label{lemma:eig0_Fourier2}
When $\tau=\beta$, the real parts of the eigenvalues of the matrix $\Psi(\xi)$ defined in \eqref{ODE_Fourier} have the following expansion when $|\xi|\to 0$:
\begin{eqnarray*}
  \func{Re}(\lambda_{1,2})(\xi) &=& -\frac{\eta^2}{2}|\xi|^4+O(|\xi|^5) \\
  \func{Re}(\lambda_{3})(\xi) &=& -|\xi|^2+O(|\xi|^3) \\
 \func{Re}(\lambda_{4})(\xi) &=& -\frac{1}{\tau}+O(|\xi|^2).
\end{eqnarray*}
In particular, all of them are negative.
\end{lemma}
\begin{pf}
To show the asymptotic behaviour of the eigenvalues in the Fourier case when $\tau=\beta$ we will follow the same steps as in Section \ref{sec:eigFourier1}. First, observe that now the characteristic polynomial \eqref{Chara_Pol_Fourier} takes the form
\begin{eqnarray}\label{charpol_Fourier_0_2}
\hspace{0.7cm}\tau p_{\mathrm{c}}(\lambda)= \tau\lambda^4+(1-\tau\zeta^2)\lambda^3+( \tau\eta^2\zeta^2-\tau-1)\zeta^2\lambda^2 +
( (\eta^2+\tau)\zeta^2-1)\zeta^2\lambda +{\zeta^4}.
\end{eqnarray}
 In this case we have the following expansion of the eigenvalues
$$\lambda_j(\zeta) =\lambda_j^0+\lambda_j^1\zeta+\lambda_j^2\zeta^2+\ldots,\ j=1,\ldots,4,\ \textrm{ for } \zeta\to 0$$
as
\begin{eqnarray*}
  \lambda_{1,2}(\zeta) &=& \pm \zeta \mp \frac{\eta^2}{2}\zeta^3-\frac{\eta^2}{2}\zeta^4+O(\zeta^5), \\
  \lambda_{3}(\zeta) &=& \zeta^2+O(\zeta^3), \\
  \lambda_{4}(\zeta) &=& -\frac{1}{\tau}+O(\zeta^2).
\end{eqnarray*}
Consequently, recall that $\zeta=i|\xi|$, then we obtain the for $|\xi|\to 0$
\begin{eqnarray*}
  \func{Re}(\lambda_{1,2})(\xi) &=& -\frac{\eta^2}{2}|\xi|^4+O(|\xi|^5), \\
  \func{Re}(\lambda_{3})(\xi) &=& -|\xi|^2+O(|\xi|^3), \\
 \func{Re}(\lambda_{4})(\xi) &=& -\frac{1}{\tau}+O(|\xi|^2).
\end{eqnarray*}
\end{pf}

\begin{lemma}\label{lemma:eiginfty_Fourier2}
When $\tau=\beta$, the real parts of the eigenvalues of the matrix $\Psi(\xi)$ defined in \eqref{ODE_Fourier} have the following expansion when $|\xi|\to \infty$:
\begin{eqnarray*}
  \func{Re}(\lambda_{1})(\xi) &=& -\frac{1}{\tau}+O(|\xi|^{-1}), \\
  \func{Re}(\lambda_{2})(\xi) &=& -\frac{1}{\eta^2}+O(|\xi|^{-1}), \\
  \func{Re}(\lambda_{3})(\xi) &=& -\frac{1-\sqrt{1-4\eta^2}}{2}|\xi|^2+O(|\xi|), \\
  \func{Re}(\lambda_{4})(\xi) &=& -\frac{1+\sqrt{1-4\eta^2}}{2}|\xi|^2+O(|\xi|).
\end{eqnarray*}
In particular, all of them are negative.
\end{lemma}
\begin{pf}
As in the proof of Lemma \ref{lemma:eiginfty_Fourier1}, and as we want the asymptotic expansion of the eigenvalues when $|\xi|\to\infty$, we take $\nu=\zeta^{-1}=(i|\xi|)^{-1}$.
Then, for $\tau=\beta$, \eqref{charpol_Fourier_infty} takes the form
\begin{eqnarray}\label{charpol_Fourier_infty_2}
&& \tau\det(L\nu^{2}- A-\mu Id)= \\
\nonumber && \tau\mu^4 + \left({\nu^2}-\tau\right)\, \mu^3 + \left(\eta^2\tau-(\tau+1)\nu^2\right) \mu^2 + (\eta^2+\tau-\nu^2)\nu^2 \mu + \nu^4.
\end{eqnarray}
As in the proof of Lemma \ref{lemma:eiginfty_Fourier1}, observe that if $\lambda(\xi)$ is an eigenvalue of \eqref{charpol_Fourier_0_2}, then
$$\mu(\nu)=\nu^2\lambda=\zeta^{-2}\lambda=-|\xi|^{-2}\lambda$$
is an eigenvalue of \eqref{charpol_Fourier_infty_2}. We can now compute the asymptotic expansion of the roots of the characteristic polynomial \eqref{charpol_Fourier_infty_2} in terms of $\nu$ when $\nu\to 0$ (that is, when $|\xi|\to \infty$) and obtain
\begin{eqnarray*}
  \mu_{1}(\nu) &=& -\frac{1}{\tau}\nu^2+O(\nu^3), \\
  \mu_{2}(\nu) &=& -\frac{1}{\eta^2}\nu^2+O(\nu^3), \\
  \mu_{3}(\nu) &=& \frac{1-\sqrt{1-4\eta^2}}{2}+O(\nu), \\
  \mu_{4}(\nu) &=& \frac{1+\sqrt{1-4\eta^2}}{2}+O(\nu).
\end{eqnarray*}
Consequently, for $|\xi|\to\infty$, we have

\begin{eqnarray*}
  \func{Re}(\lambda_{1})(\xi) &=& -\frac{1}{\tau}+O(|\xi|^{-1}), \\
  \func{Re}(\lambda_{2})(\xi) &=& -\frac{1}{\eta^2}+O(|\xi|^{-1}), \\
  \func{Re}(\lambda_{3})(\xi) &=& -\frac{1-\sqrt{1-4\eta^2}}{2}|\xi|^2+O(|\xi|), \\
  \func{Re}(\lambda_{4})(\xi) &=& -\frac{1+\sqrt{1-4\eta^2}}{2}|\xi|^2+O(|\xi|).
\end{eqnarray*}
Observe that in all the cases we have $\func{Re}(\lambda_{j})(\xi)<0$, $j=1,\ldots,4$.  
\end{pf}
\section{The Cattaneo-law model}\label{Section_Cattaneo}

%

In this section, we consider the standard linear solid model for viscoelasticity coupled with heat conduction of Cattaneo type (see Section \ref{Introduction} for more details). That is, we consider system \eqref{Rivera_Model_Cattaneo} with $\tau_0>0$. Again, without loss of generality, we now take $a=1$ and $\gamma=\kappa$.

\subsection{Functional setting and well-posedness of the Cattaneo-law model}\label{Well-posed-Cattaneo}
The functional setting and the proof for the well-posedness in the Cattaneo case will follow the same ideas as the Fourier one (see Section \ref{Well-posed-Fourier}).

As in the Fourier problem, the first thing we need is to write problem \eqref{Rivera_Model_Cattaneo} as a first-order evolution equation. By taking $v=u_t$ and $w=u_{tt}$, the previous system can be written as
\begin{equation}\label{eq:ev_equation_C}
\left\{
\begin{array}{l}
\dfrac{d}{dt}U(t) = \AC U(t) ,\ \ t\in[0,+\infty)\vspace{0.2cm} \\
U(0)=U_0\end{array}
\right.
\end{equation}
where $U(t)=(u,v,w,\theta,q)^T$, $U_0=(u_0,u_1,u_2,\theta_0,q_0)^T$ and $\AC:\D(\AC)\subset \cH_{\mathrm{C}}\longrightarrow \cH_{\mathrm{C}}$ is the following linear operator
\begin{equation*}
\AC \left(
\begin{array}{c}
u \\
v \\
w \\
\theta \\
q
\end{array}%
\right)=
\left(
\begin{array}{c}
v \\
w \\
\frac{1}{\tau}\Delta(u+\beta v-\eta \theta)-\frac{1}{\tau}w \\
\eta\Delta(v+\tau w )-\gamma \nabla\cdot q \\
-\frac{1}{\tau_0}(q+\kappa \nabla\theta)
\end{array}%
\right)
\end{equation*}
(the dot stands for the usual inner product in $\R^N$). Following \cite{ABFRS_2012} and \cite{PellSaid_2016}, we introduce the energy space
$$\cH_{\mathrm{C}}= H^1(\R^N)\times H^1(\R^N)\times L^2(\R^N)\times L^2(\R^N) \times (L^2(\R^N))^N$$
with the following inner product
\begin{equation*}
\begin{array}{ll}
\langle (u,v,w,\theta,q),\,(u_1,v_1,w_1,\theta_1,q_1)  \rangle_{\cH_{\mathrm{C}}} & =  \langle (u,v,w,\theta),\,(u_1,v_1,w_1,\theta_1)  \rangle_{\cH_{\mathrm{F}}} + \,\tau_0\displaystyle\int_{\mathbb{R}^N} q\cdot \overline{q}_1 \,dx \\
& = \tau(\beta-\tau) \displaystyle\int_{\mathbb{R}^N}\nabla v\cdot\nabla \overline{v}_1  \,dx +\int_{\mathbb{R}^N} \nabla(u+\tau v)\cdot \nabla(\overline{u}_1+\tau \overline{v}_1)\,dx  \vspace{0.2cm}\\
& + \,\displaystyle\int_{\mathbb{R}^N} (v+\tau w)\, (\overline{v}_1+\tau \overline{w}_1) \,dx
 \, + \,\displaystyle\int_{\mathbb{R}^N} (u+\tau v)\, (\overline{u}_1+\tau \overline{v}_1) \,dx  \vspace{0.2cm}\\
&  + \,\displaystyle\int_{\mathbb{R}^N} v\, \overline{v}_1 \,dx + \,\displaystyle\int_{\mathbb{R}^N} \theta\, \overline{\theta}_1 \,dx + \,\dfrac{\gamma}{\kappa}\tau_0\displaystyle\int_{\mathbb{R}^N} q\cdot \overline{q}_1 \,dx
\end{array}
\end{equation*}
and the corresponding norm
\begin{equation*}\label{normC}
\begin{array}{ll}
\left\| (u,v,w,\theta,q)\right\|^2_{\cH_{\mathrm{C}}} = & \tau(\beta-\tau) \|\nabla v\|^2_{L^2(\R^N)}  + \|\nabla(u+\tau v)\|^2_{L^2(\R^N)} +  \|v+\tau w\|^2_{L^2(\R^N)}  \vspace{0.2cm} \\
& + \|u+\tau v\|^2_{L^2(\R^N)} + \|v\|^2_{L^2(\R^N)}+ \|\theta\|^2_{L^2(\R^N)} +\dfrac{\gamma}{\kappa}\tau_0 \|q\|^2_{(L^2(\R^N))^N},
\end{array}
\end{equation*}
where 
$$\|q\|^2_{(L^2(\R^N))^N} = \sum_{i=1}^N \|q_i\|^2_{L^2(\R^N)}.$$ 
As in Remark \ref{rem:equivnorms}, it can be seen that this is an equivalent norm to the natural one in $\cH_{\mathrm{C}}$.

We consider \eqref{eq:ev_equation_C} in the Hilbert space $\cH_{\mathrm{C}}$ with domain
\begin{eqnarray*}\label{eq:domain_F}
\D(\AC) =\left\{ (u,v,w,\theta,q)\in \cH_{\mathrm{C}}; w,\theta\in H^1(\R^N), \, q\in (H^1(\R^N))^N \right.
\\ \left.\textrm{ and } u+\beta v-\eta\theta, \eta v+\tau\eta w+\theta \in H^2(\R^N) \right\}.
\end{eqnarray*}

\begin{theorem}\label{thm:semigroup_C}
Under the dissipative condition $0<\tau\leq\beta$, the operator $\AC$ is the generator of a $C_0$-semigroup of contractions on $\cH_{\mathrm{C}}$. In particular, for any $U_0\in \D(\AC)$, there exists a unique solution $U\in C^1([0,+\infty);\cH_{\mathrm{C}})\cap  C([0,+\infty);\D(\AC))$ satisfying \eqref{eq:ev_equation_C}.
\end{theorem}

\begin{pf}
The proof follows the same steps as the proof of Theorem \ref{thm:semigroup_F}. As in Theorem \ref{thm:semigroup_F}, we are going to work with a perturbed problem, which now is

\begin{equation*}\label{eq:perturbed_equation_C}
\left\{
\begin{array}{l}
\dfrac{d}{dt}U(t) = \AbC U(t) ,\ \ t\in[0,+\infty)\vspace{0.2cm} \\
U(0)=U_0\end{array}
\right.
\end{equation*}
where
$$\AbC \left(
\begin{array}{c}
u \\
v \\
w \\
\theta \\
q
\end{array}\right) =
(\AC +B_{\mathrm{C}}) \left(
\begin{array}{c}
u \\
v \\
w \\
\theta\\
q
\end{array}\right) =
\left(
\begin{array}{c}
v \\
w \\
\frac{1}{\tau}\Delta(u+\beta v-\eta \theta)-\frac{1}{\tau}w- \frac{1}{\tau} u - v - \frac{1}{\tau^2} v \\
\eta\Delta( v+\tau w )-\gamma\nabla\cdot q \\
-\frac{1}{\tau_0}(q+\kappa\nabla\theta)
\end{array}
\right).$$
Using the same ideas as in Theorem \ref{thm:semigroup_F}, we can see that, for any $U\in \D(\AbC)$,
$$\Re \langle \AbC U,U \rangle_{\cH_{\mathrm{C}}} = -(\beta-\tau) \displaystyle\int_{\mathbb{R}^N}|\nabla v|^2\, dx   -\frac{1}{\tau} \displaystyle\int_{\mathbb{R}^N}| v|^2\, dx - \displaystyle\int_{\mathbb{R}^N}| q|^2\, dx\leq 0$$
since $0<\tau\leq \beta$. Hence, the operator $\AbC$ is dissipative when $0<\tau\leq \beta$.

We are now going to prove that $(\lambda Id-\AbC)$ is surjective in $\cH_{\mathrm{C}}$ for
\begin{equation}\label{lambda}
\lambda_\ast=\dfrac{-1+\sqrt{1+4\tau}}{2\tau}>0,
\end{equation}
which means that $\AbC$ is maximal in $\cH_{\mathrm{C}}$. As we said in the proof of Theorem \ref{Well-posed-Fourier}, we are going to give here the complete details of this proof, which can be used as a proof for the maximality of $\AbF$ in $\cH_{\mathrm{F}}$ just taking $\tau_0=0$.  We consider a given $F=(f,g,h,j,p)\in\cH_{\mathrm{C}}$. We want to see that there exists a unique $U=(u,v,w,\theta,q)\in\D(\AC)$ such that $(\lambda Id-\AbC)U=F$, that is
\begin{eqnarray}\label{eq:maximal_C}
\lambda u-v  = & f \in H^1(\R^N) \nonumber\\
\lambda v-w  = & g \in H^1(\R^N) \nonumber\\
\lambda w-\dfrac{1}{\tau}\Delta(u+\beta v-\eta\theta) + \dfrac{1}{\tau}w  + \dfrac{1}{\tau}u + v+\dfrac{1}{\tau^2}v  = & h \in L^2(\R^N)\label{maximalC_1}\\
\lambda\theta-\eta\Delta (v+\tau w)+\gamma\nabla\cdot q  = & j \in L^2(\R^N)\label{maximalC_2}\\
(\lambda\tau_0 +1)q +\kappa\nabla\theta = & \tau_0 p \in (L^2(\R^N))^N \label{maximalC_3}.
\end{eqnarray}
By taking $v=\lambda u-f$, $w=\lambda v-g=\lambda^2u-\lambda f-g$ and plugging it into \eqref{maximalC_1}, \eqref{maximalC_2} and \eqref{maximalC_3} we arrive at the following system:
\begin{eqnarray}\label{eq:aux_C}
&&\left(\lambda^3\tau+\lambda^2+ \lambda\dfrac{\tau^2+1}{\tau}+1\right)u-\left(1+\lambda\beta\right)\Delta u +\eta\Delta\theta  \nonumber \\
&&= \Big(\lambda^2\tau+\lambda+\dfrac{\tau^2+1}{\tau}\Big)f  -\beta\Delta f+ \left(\lambda \tau+1\right) g +\tau h -\lambda\eta(1+\tau\lambda)\Delta u +\lambda\theta+\gamma\nabla\cdot q \label{eq:aux1}\\
 &&=j-\eta(1+\lambda\tau)\Delta f-\tau\eta \Delta g + \dfrac{\gamma}{\kappa}(\lambda\tau_0+1)q+\gamma\nabla\theta\label{eq:aux2}\\
 && =\dfrac{\gamma}{\kappa}\tau_0 p.\label{eq:aux3}
\end{eqnarray}
In order to solve this problem, we consider its weak formulation given by the bilinear form
$$\mathcal{M}: (H^1(\R^N)\times H^1(\R^N)\times (H^1(\R^N))^N)\times (H^1(\R^N)\times H^1(\R^N)\times (H^1(\R^N))^N)\longrightarrow \R$$
given by
\begin{eqnarray*}
\mathcal{M}\left((u,\theta,q),(\varphi,\chi,\phi)\right) 
&=& \left(\lambda^3\tau+\lambda^2+ \lambda\dfrac{\tau^2+1}{\tau}+1\right) \displaystyle\int_{\mathbb{R}^N} u\varphi \,dx \\
&&+ (1+\beta\lambda)\displaystyle\int_{\mathbb{R}^N}\nabla u\cdot\nabla\varphi\,dx -\eta \displaystyle\int_{\mathbb{R}^N}\nabla \theta\cdot\nabla\varphi\,dx  \\
 &&+\lambda\eta(1+\tau\lambda) \displaystyle\int_{\mathbb{R}^N} \nabla u\cdot\nabla\chi\,dx + \lambda\displaystyle\int_{\mathbb{R}^N} \theta\chi\,dx +\gamma \displaystyle\int_{\mathbb{R}^N} (\nabla\cdot q)\chi\,dx \\
 &&+(\tau_0\lambda+1) \displaystyle\int_{\mathbb{R}^N} q\phi \,dx +\kappa \displaystyle\int_{\mathbb{R}^N} \nabla\theta\cdot\phi\,dx
\end{eqnarray*}
and the linear one
$$\mathcal{K}: (H^1(\R^N)\times H^1(\R^N)\times (H^1(\R^N))^N)\longrightarrow \R$$
given by
\begin{eqnarray*}
\mathcal{K}(\varphi,\chi,\phi)& =& \left(\lambda^2\tau+\lambda+\dfrac{\tau^2+1}{\tau}\right)\displaystyle\int_{\mathbb{R}^N} f\varphi\,dx +\beta \displaystyle\int_{\mathbb{R}^N} \nabla f\cdot\nabla\varphi\,dx \\
&&+ (\lambda \tau+1) \displaystyle\int_{\mathbb{R}^N}g\varphi\,dx + \tau \displaystyle\int_{\mathbb{R}^N} h\varphi\,dx+ \displaystyle\int_{\mathbb{R}^N}j\chi\,dx\\
&& +\eta(1+\tau\lambda)\displaystyle\int_{\mathbb{R}^N}\nabla f\cdot\nabla\chi\,dx + \tau\eta\displaystyle\int_{\mathbb{R}^N} \nabla g\cdot\nabla\chi\,dx +\tau_0 \displaystyle\int_{\mathbb{R}^N} p\cdot\phi\,dx.
\end{eqnarray*}
We can see that $\mathcal{M}$ is coercive. Integrating by parts using that $u,\theta\in H^1(\R^N)$ and $q\in (H^1(\R^N))^N$,  we obtain:
\begin{eqnarray*}
\mathcal{M}\left((u,\theta,q),(u,\theta,q)\right)  
&= &\left(\lambda^3\tau+\lambda^2+ \lambda\dfrac{\tau^2+1}{\tau}+1\right) \displaystyle\int_{\mathbb{R}^N} u^2 \,dx \\
&&+ (1+\beta\lambda)\displaystyle\int_{\mathbb{R}^N}|\nabla u|^2\,dx \\
 &&+\eta (\tau\lambda^2+\lambda-1) \displaystyle\int_{\mathbb{R}^N} \nabla u\cdot \nabla \theta  + \lambda \displaystyle\int_{\mathbb{R}^N}\theta^2\,dx + \dfrac{\gamma}{\kappa}(\lambda\tau_0+1)  \displaystyle\int_{\mathbb{R}^N}|q|^2\,dx.  \\
\end{eqnarray*}
Taking $\lambda_*=\frac{-1+\sqrt{1+4\tau}}{2\tau}>0$ (see \eqref{lambda}) the third term vanishes and hence,
\begin{equation*}
\mathcal{M}\left((u,\theta,q),(u,\theta,q)\right) \geq C \left(\|u\|^2_{H^1(\R^N)} +\|\theta\|^2_{H^1(\R^N)} + \|q\|_{(H^1(\R^N))^N} \right)
\end{equation*}
for $C=\min\left\{\lambda_*^3\tau+\lambda_*^2+ \lambda_*\dfrac{\tau^2+1}{\tau}+1,\ 1+\beta\lambda_*,\ \lambda_*, \ \lambda_*\tau_0+1\right\}$.  

We can also see that $\mathcal{M}$ is bounded using the H\"{o}lder inequality:
\begin{equation*}
\begin{array}{ll}
\mathcal{M}\left((u,\theta,q),(\varphi,\chi,\phi)\right) & \leq \left|\lambda^3\tau+\lambda^2+ \lambda\dfrac{\tau^2+1}{\tau}+1\right| \|u\|_{L^2(\R^N)}\cdot\|\varphi\|_{L^2(\R^N)}  \\ \\
& + |1+\beta\lambda| \|\nabla u\|_{L^2(\R^N)}\|\nabla \varphi\|_{L^2(\R^N)} + \eta \| \nabla\theta\|_{L^2(\R^N)}\| \nabla\varphi\|_{L^2(\R^N)}\\ \\
& + \lambda\eta |1+\tau\lambda| \| \nabla u\|_{L^2(\R^N)} \| \nabla\chi\|_{L^2(\R^N)}\| +\lambda \| \theta\|_{L^2(\R^N)} \| \chi\|_{L^2(\R^N)} \\ \\
& + \kappa \| q\|_{(L^2(\R^N))^N}\| \nabla\chi\|_{L^2(\R^N)} + \\ \\
& + \left|\dfrac{\gamma}{\kappa}\lambda\tau_0+1\right| \| q\|_{(L^2(\R^N))^N}\| \phi\|_{(L^2(\R^N))^N}+ \gamma \| \nabla\theta\|_{L^2(\R^N)}\| \phi\|_{(L^2(\R^N))^N}. 
\end{array}
\end{equation*}
Hence,
\begin{eqnarray*}
\mathcal{M}\left((u,\theta,q),(\varphi,\chi,\phi)\right)&\leq& 
 C \left(\|u\|^2_{H^1(\R^N)} +\|\theta\|^2_{H^1(\R^N)} + \|q\|_{(H^1(\R^N))^N} \right)^{1/2} \\
 &&\times\left(\|\varphi\|^2_{H^1(\R^N)} +\|\chi\|^2_{H^1(\R^N)} + \|\phi\|_{(H^1(\R^N))^N} \right)^{1/2}
\end{eqnarray*}
for some $C>0$ (for any $\lambda>0$ and, in particular, for $\lambda=\lambda_*>0$).

Finally, as $f,g\in H^1(\R^N)$, $h\in L^2(\R^N)$,  $p\in (L^2(\R^N)^N$, we can see that $\mathcal{K}$ is bounded:
\begin{equation*}
\begin{array}{ll}
\mathcal{K}(\varphi,\chi,\phi) & \leq |\lambda^2\tau+\lambda+\tau+\frac{1}{\tau}| \| f\|_{L^2(\R^N)} \| \varphi\|_{L^2(\R^N)} + |\beta| \| \nabla f\|_{L^2(\R^N)}\| \nabla \varphi\|_{L^2(\R^N)} \\ \\
& + \left(|\tau| \| h\|_{L^2(\R^N)} +|\tau\lambda+1|\| g\|_{L^2(\R^N)} + \| j\|_{L^2(\R^N)} \right) \|\varphi\|_{L^2(\R^N)} \\ \\
& + \left(| \eta (1+\lambda\tau) | \| \nabla f\|_{L^2(\R^N)} + |\tau\eta|\| \nabla g\|_{L^2(\R^N)}  \right)\| \nabla \chi\|_{L^2(\R^N)} \\ \\
& + |\tau_0| \| p\|_{(L^2(\R^N))^N} \| \phi\|_{(L^2(\R^N))^N} \\ \\
& \leq C \left( \| \varphi\|^2_{H^1(\R^N)} + \| \chi\|^2_{H^1(\R^N)} + \| \phi\|^2_{(H^1(\R^N))^N}   \right) ^{1/2}
\end{array}
\end{equation*}
for some $C>0$ (for any $\lambda>0$ and, in particular, for $\lambda=\lambda_*>0$). So, by the Lax-Milgram theorem, there exists a unique $(u,\theta,q)\in H^1(\R^N)\times H^1(\R^N)\times (H^1(\R^N))^N$ such that

$$\mathcal{M}\left((u,\theta,q),(\varphi,\chi,\phi)\right)=\mathcal{K}(\varphi,\chi,\phi)\ \ \forall\ (\varphi,\chi,\phi)\in H^1(\R^N)\times H^1(\R^N)\times (H^1(\R^N))^N.$$
By the same arguments as in \cite{ABFRS_2012}, this mild solution is, indeed, a strong solution of \eqref{eq:aux_C}-\eqref{eq:aux3}. This means the operator $\lambda_* Id-\AbC$ is maximal in $\cH_{\mathrm{C}}$, as we claimed.
Finally, by the Lummer-Phillips theorem we can conclude that $\AbC$ is the generator of a $C_0$-semigroup of contractions in $\cH_{\mathrm{C}}$. And so it is $\AC$, as it is a bounded perturbation of $\AbC$ in $\cH_{\mathrm{C}}$. The regularity of the solution is a consequence of this fact (see \cite{Brezis} or \cite{Pazy}).
\end{pf}

\subsection{Stability results: the case $0<\tau<\beta$ }\label{stabCattaneo_1}
We show some stability results  of a norm related
with the solution of \eqref{Bose_Problem} when $\tau_0>0$. We discuss the two cases separatedly: $0<\tau<\beta$ in the present section and $0<\tau= \beta$ in Section \ref{stabCattaneo_2}. In this section, we show that the $L^2$-norm of the norm related to the solution decays with the rate $(1+t)^{-N/4}$. But, unlike the Fourier model, the Cattaneo model yields a regularity loss. 

Using the same change of variables as in Section \ref{Section_Fourier}, we may rewrite system \eqref{Bose_Problem} as

\begin{subequations}
\begin{eqnarray}
&&u_{t}-v=0, \label{First_equation_system_Cattaneo}\\
&&v_t-w=0,
 \label{Second_equation_system_Cattaneo}\\
&&\tau w_{t}+ w-\Delta u-\beta\Delta v+\eta\Delta \theta=0,  \label{Third_equation_system_Cattaneo} \\
&&\theta _{t}+\kappa \nabla\cdot q -\eta\tau\Delta w-\eta\Delta v=0,  \label{Fourth_equation_system_Cattaneo}\\
&& \tau_0q_t+q+\kappa \nabla\theta=0,\label{Fifth_equation_system_Cattaneo}
\end{eqnarray}%
with the initial data
\begin{eqnarray}
&&u(x,0)=u_0(x),\quad v\left( x,0\right)  =v_{0}\left( x\right),\quad w(x,0)=w_0(x),\notag\\
&& \theta (x,0)=\theta_0(x),\quad q(x,0)=q_0.
\label{Initial_condition_system_Cattaneo}
\end{eqnarray}
\label{Main_First_Order_system_Cattaneo}
\end{subequations}
Taking the Fourier transform of the previous system \eqref{Main_First_Order_system_Cattaneo}, we obtain
\begin{subequations}
\begin{eqnarray}
&&\hat{u}_{t}-\hat{v}=0, \label{First_equation_system_Fourier_Cattaneo}\\
&&\hat{v}_{t}-\hat{w}=0,
 \label{Second_equation_system_Fourier_Cattaneo}\\
&&\tau\hat{w}_{t}+\hat{w}+|\xi|^2\hat{u}+\beta |\xi|^2\hat{v}-\eta|\xi|^2\hat{\theta}=0,  \label{Third_equation_system_Fourier_Cattaneo} \\
&&\hat{\theta} _{t}+\kappa i\xi\cdot \hat{q}+\eta\tau|\xi|^2 \hat{w}+\eta |\xi|^2\hat{v}=0,  \label{Fourth_equation_system_Fourier_Cattaneo}
\\
&& \tau_0\hat{q}_t+\hat{q}+i\xi \kappa \hat{\theta}=0,\label{Fifth_equation_system__Fourier_Cattaneo}
\end{eqnarray}
with the initial data
\begin{eqnarray}
&&\hat{u}\left( \xi,0\right)  =\hat{u}_{0}\left( \xi\right),\quad \hat{v}\left( \xi,0\right) =v_{0}\left( \xi\right),\quad
\hat{w}\left( \xi,0\right)  =\hat{w}_{0}\left( \xi\right),\notag \\
&& \hat{\theta}\left( \xi,0\right)=\hat{\theta}(\xi),\quad \hat{q}\left( \xi,0\right)=\hat{q}(\xi).
\label{Initial_condition_system_Fourier_Cattaneo}
\end{eqnarray}
\label{System_New_2_Cattaneo}
\end{subequations}

\subsubsection{The Lyapunov functional}
Our main result in this section follows the same ideas as in the corresponding sections of the Fourier problem, and reads as:
\begin{theorem} \label{Decay_Second_system_Cattaneo}
Assume that $0<\tau<\beta$.  Let $\VC(x,t)=(\tau u_{tt}+u_t,\nabla (\tau u_t+u),\nabla u_t,\theta,q)^T(x,t)$, where $(u,\theta,q)$ is the solution of (\ref{Bose_Problem}) with $\tau_0>0$. Let $s$ be a nonnegative integer and let $\VC^0=\VC(x,0)\in
H^{s}\left(
\mathbb{R}^N
\right) \cap L^{1}\left(
\mathbb{R}^N
\right) .$  Then, the following decay estimate
\begin{equation*}\label{Main_estimate_Theorem_2}
\left\Vert \nabla^{k}\VC\left( t\right) \right\Vert _{L^{2}(\mathbb{R}^N)}\leq  C(1+t)^{-N/4-k/2}\left\Vert \VC^0 \right\Vert _{L^{1}(\mathbb{R}^N)}+C(1+t)^{-\ell/2}\left\Vert \nabla^{k+\ell}\VC^0 \right\Vert _{L^{2}(\mathbb{R}^N)}
\end{equation*}
holds,
for any $t\geq 0$ and $0\leq k+\ell\leq s$, where $C$ and $c$ are two positive constants independent of $t$ and $\VC^0$.
\end{theorem}
\begin{remark}
Observe that we have the same decay rate as in the Fourier case when $0<\tau<\beta$ but, in that case, there was not any regularity loss phenomenon (although the norm related to the solution is not exactly the same in both cases). Hence, the fact that the heat coupling is of Cattaneo type does not imply a slower decay rate, but a loss of regularity in the solution. Such regularity loss is one of the major difficulties for proving global existence in nonlinear problems. In many cases higher regularity assumption is needed to close the estimates. See for instance  \cite{HKa06} and \cite{Racke_Said_2012_1}. 
\end{remark}

To justify rigorously the decay rate \eqref{Main_estimate_Theorem_2} without computing the solution we use the energy method in the Fourier space to build an appropriate Lyapunov functional that will give us the decay rate of the Fourier image of $\VC(x,t)$ and, hence, the above decay rate. In Section \ref{Section_Asym_Cattaeo}, we use eigenvalues expansion to show that the decay rate obtained in Theorem \ref{Decay_Second_system_Cattaneo} is optimal.

We will follow the same ideas as in Section \ref{sec:FourierFunctional}. The decay rate of the Fourier image of $W$ is the one given in the following proposition.

\begin{proposition}
\label{Main_Lemma_Cattaneo} Assume that $0<\tau<\beta$. Let $\VC(x,t)=(\tau u_{tt}+u_t,\nabla (\tau u_t+u),\nabla u_t,\theta,q)^T(x,t)$, where $(u(x,t),\theta(x,t),q(x,t))$ is the solution of \eqref{Bose_Problem} for $\tau_0>0$. Then we have
\begin{equation}  \label{V_Inequality}
|\hat{\VC}(\xi,t)|^2\leq Ce^{-c\rC(\xi)t}|\hat{\VC}(\xi,0)|^2,
\end{equation}
where
\begin{equation*}
\rC(\xi)=\frac{|\xi|^2}{1+|\xi|^2+|\xi|^4},
\end{equation*}
 and $C$ and $c$ are two positive
constants.
\end{proposition}
The proof of this  the above proposition will be done using the following lemmas, which are similar to those in Section \ref{sec:FourierFunctional}.

\begin{lemma}
\label{dissipa_Energy_Cattaneo} Assume that $0<\tau<\beta$. Let $(\hat{u},\hat{v},\hat{w},\hat{\theta},\hat{q})(\xi,t)$ be the solution of (\ref{System_New_2_Cattaneo}). The energy functional associated to this system is defined as:
\begin{equation}\label{Energy_functional_Cattaneo}
\hat{\EC}(\xi,t):=\frac{1}{2}\left\{ | \hat{v}+\tau\hat{w}|^2+|\xi|^2|\hat{u}+\tau\hat{v}|^2+\tau(\beta-\tau)|\xi|^2|\hat{v}|^2+|\hat{\theta}|^2+\tau_0|\hat{q}|^2\right\}.
\end{equation}
Then for all $t\geq 0$,
$$\hat{\EC}(\xi, t)\geq 0$$
and there exist
two positive constants $d_1$ and $d_2$ such that
\begin{equation}  \label{Positivity_Energy_Cattaneo}
d_1|\hat{\VC}(\xi,t)|^2\leq \hat{\EC}(\xi,t)\leq d_2 |\hat{\VC}(\xi,t)|^2
\end{equation}
and
\begin{equation}  \label{dE_dt_1_Cattaneo}
\frac{d}{dt}\hat{\EC}(\xi, t)=-(\beta-\tau)|\xi|^2|\hat{v}|^2-|\hat{q}|^2,
\end{equation}
where
\begin{equation*}
|\hat{\VC}(\xi, t)|^2=| \hat{v}+\tau\hat{w}|^2+|\xi|^2|\hat{u}+\tau\hat{v}|^2+|\xi|^2|\hat{v}|^2+|\hat{\theta}|^2+|\hat{q}|^2.
\end{equation*}
\end{lemma}

The proof of this Lemma \ref{dissipa_Energy_Cattaneo} is similar to the one of
Lemma \ref{dissipa_Energy}. So, we omit the details.

\begin{proof}[Proof of Proposition \ref{Main_Lemma_Cattaneo}]
Using the same method as in the proof of Proposition \ref{Main_Lemma}, we observe
first that the estimates (\ref{dF_1_dt}) for $\F_1(\xi,t)$ and (\ref{dF_2_dt_1})  for $\F_2(\xi,t)$
remain valid for the new system (\ref{Main_First_Order_system_Cattaneo}). Also, \eqref{dF_2_dt_1} can be also estimated as
\begin{eqnarray}\label{dF_2_dt_1_Cattaneo}
&&\frac{d}{dt}\F_2(\xi,t)+(1-\varepsilon_1)|\hat{v}+\tau \hat{w}|^2\notag \\
&\leq&
C(\varepsilon_0,\varepsilon_1,\varepsilon_2)(1+|\xi|^2+|\xi|^4)|\hat{v}|^2+\varepsilon_0|\hat{\theta}|^2+\varepsilon_2\frac{|\xi|^2}{1+|\xi|^2}|\hat{u}+\tau \hat{v}|^2
\end{eqnarray}
for any $\varepsilon_0,\varepsilon_1,\varepsilon_2>0$.

The previous estimate comes from the fact that:
\begin{eqnarray*}\label{F_2_Cattaneo}
&&\frac{d}{dt}\F_2(\xi,t)+|\hat{v}+\tau \hat{w}|^2-\tau (\beta-\tau )|\xi|^2|\hat{v}|^2\notag\\
&=&\tau |\xi|^2\func{Re}\left\{(\hat{u}+\tau \hat{v})\bar{\hat{v}}\right\}+\func{Re}\left\{(\hat{v}+\tau \hat{w})\bar{\hat{v}}\right\}-\eta \tau |\xi|^2\func{Re}(\hat{\theta}\bar{\hat{v}})
\end{eqnarray*}
(see Lemma \ref{Lemma_F_2}). Using Young's inequality, we have
\begin{eqnarray*}
|\tau |\xi|^2\func{Re}\left\{(\hat{u}+\tau \hat{v})\bar{\hat{v}}\right\}|\leq \varepsilon_2 \frac{|\xi|^2}{1+|\xi|^2} |\hat{u}+\tau \hat{v}|^2+C(\varepsilon_2) |\xi|^2(1+|\xi|^2)|\hat{v}|^2
\end{eqnarray*}
 and
 \begin{equation*}
|-\eta \tau |\xi|^2\func{Re}(\hat{\theta}\bar{\hat{v}})|\leq \varepsilon_0 |\hat{\theta}|^2+C(\varepsilon_0)|\xi|^4 |\hat{v}|^2
\end{equation*}
and
\begin{equation*}
|\func{Re}\left\{(\hat{v}+\tau \hat{w})\bar{\hat{v}}\right\}|\leq \varepsilon_1 |(\hat{v}+\tau \hat{w})|^2+C(\varepsilon_1) |\hat{v}|^2.
\end{equation*}
Collecting the above estimates, we obtain \eqref{dF_2_dt_1_Cattaneo}.


Our goal now is to build dissipative terms for $|\hat{\theta}|^{2}$. Indeed,
multiplying equation (\ref{Fifth_equation_system__Fourier_Cattaneo}) by $-i\xi %
\hat{\theta}$ we get%
\begin{equation}\label{q_dot_product}
-\tau_0\langle \hat{q}_t,i\xi%
\hat{\theta}\rangle-\langle\hat{q}, %
i\xi\hat{\theta}\rangle+\kappa|\xi|^2  |\hat{\theta}|^2=0,
\end{equation}
where $\langle \cdot,\cdot\rangle$ is the  dot product in $\C^n$. Multiplying  (\ref{Fourth_equation_system_Fourier_Cattaneo}) by $%
\tau _{0}i\xi\cdot  \bar{\hat{q}}=\tau _{0}\langle i\xi,\hat{q}\rangle $, we get
\begin{equation}\label{theta_dot_product}
\tau _{0}\langle i\xi\hat{\theta} _{t},  \hat{q}\rangle-\kappa\tau_0 |\xi\cdot \hat{q}|^2+\eta\tau \tau _{0}|\xi|^2 \langle i\xi\hat{w},  \hat{q}\rangle+\tau _{0}\eta |\xi|^2\langle i\xi \hat{v},  \hat{q}\rangle.
\end{equation}
Summing up \eqref{q_dot_product} and \eqref{theta_dot_product}, and taking the real part, we get
\begin{equation}\label{dF_3_dt}
\frac{d}{dt}\tau _{0}\func{Re}\langle i\xi\hat{\theta},  \hat{q}\rangle+\kappa|\xi|^2  |\hat{\theta}|^2=\tau_0\kappa |\xi\cdot \hat{q}|^2+\func{Re}\langle\hat{q}, i\xi %
+\hat{\theta}\rangle
-\eta\tau_0|\xi|^2\func{Re} \langle i\xi(\hat{v}+\tau \hat{w}) ,  \hat{q}\rangle.
\end{equation}
 Now, by using first \eqref{First_equation_system_Fourier_Cattaneo} and \eqref{Second_equation_system_Fourier_Cattaneo} and, then, \eqref{Fifth_equation_system__Fourier_Cattaneo}, we have
 \begin{eqnarray*}
\langle i\xi(\hat{v}+\tau \hat{w}),  \hat{q}\rangle&=&\langle  i\xi(\hat{u}_t+\tau \hat{v}_t),  \hat{q}\rangle\notag\\
&=&\frac{d}{dt} \langle i\xi(\hat{u}+\tau \hat{v}),  \hat{q}\rangle\notag- \langle i\xi(\hat{u}+\tau \hat{v}),  \hat{q}_t\rangle\notag\\
&=& \frac{d}{dt} \langle i\xi(\hat{u}+\tau \hat{v}),  \hat{q}\rangle\notag+\frac{1}{\tau_0} \langle i\xi(\hat{u}+\tau \hat{v}),\hat{q}\rangle +\frac{\kappa}{\tau_0} \langle i\xi(\hat{u}+\tau \hat{v}),i\xi\hat{\theta}\rangle.
\end{eqnarray*}
(recall that $\tau_0>0$). Plugging these into \eqref{dF_3_dt}, we get

\begin{eqnarray}\label{dF_3_dt_2}
\frac{d}{dt}\mathcal{F}_3(\xi,t)+\kappa|\xi|^2  |\hat{\theta}|^2&=&\tau_0\kappa |\xi\cdot \hat{q}|^2+\func{Re}\langle\hat{q}, i\xi %
\hat{\theta}\rangle\notag\\
&&
-\eta|\xi|^2 \func{Re}\langle i\xi(\hat{u}+\tau \hat{v}),\hat{q}\rangle
 -\kappa\eta|\xi|^2\func{Re} \langle i\xi(\hat{u}+\tau \hat{v}),i\xi\hat{\theta}\rangle
\end{eqnarray}
where
\begin{equation*}
\mathcal{F}_3(t)=\tau _{0}\func{Re}\langle i\xi\hat{\theta},  \hat{q}\rangle+\eta\tau_0|\xi|^2\func{Re} \langle i\xi(\hat{u}+\tau \hat{v}),  \hat{q}\rangle.
\end{equation*}
Now, we see that
\begin{equation*}
\langle i\xi(\hat{u}+\tau \hat{v}),i\xi\hat{\theta}\rangle = ( \hat{u}+\tau \hat{v})\bar{\hat{\theta} }\langle i\xi,i\xi \rangle = (\hat{u}+\tau \hat{v})\bar{\hat{\theta} } |\xi|^2.
\end{equation*}
Consequently, \eqref{dF_3_dt_2} takes the form

\begin{eqnarray}\label{dF_3_dt_3}
\frac{d}{dt}\mathcal{F}_3(\xi,t)+\kappa|\xi|^2  |\hat{\theta}|^2&=&\tau_0\kappa |\xi\cdot \hat{q}|^2+\func{Re}\langle\hat{q}, i\xi %
\hat{\theta}\rangle\notag\\
&&
-\eta|\xi|^2 \func{Re}\langle i\xi(\hat{u}+\tau \hat{v}),\hat{q}\rangle
 -\kappa\eta|\xi|^4\func{Re}( ( \hat{u}+\tau \hat{v})\bar{\hat{\theta} }).
\end{eqnarray}
Now, we define the functional
\begin{equation*}
\mathcal{F}_4(\xi, t)=\kappa |\xi|^2F_1(\xi,t)+\mathcal{F}_3(\xi,t).
\end{equation*}
 Then, we have from \eqref{F_1_main} and \eqref{dF_3_dt_3}

  \begin{eqnarray}\label{dF_4_dt_1}
  &&\frac{d}{dt}\mathcal{F}_4(\xi,t)+\kappa|\xi|^2  |\hat{\theta}|^2+\kappa|\xi|^4|\hat{u}+\tau \hat{v}|^2-\kappa |\xi|^2|\hat{v}+\tau \hat{w}|^2\notag\\
&=&\tau_0\kappa |\xi\cdot \hat{q}|^2+\func{Re}\langle\hat{q}, i\xi %
\hat{\theta}\rangle-\eta|\xi|^2 \func{Re}\langle i\xi(\hat{u}+\tau \hat{v}),\hat{q}\rangle\notag\\
&&+\kappa |\xi|^4(\tau -\beta)\func{Re}(\hat{v}(\bar{\hat{u}}+\tau \bar{\hat{v}})).
\end{eqnarray}
Using Cauchy--Schwarz  inequality together with Young's inequality, we have
\begin{eqnarray*}
&&|\xi\cdot \hat{q}|^2\leq |\xi|^2|\hat{q}|^2\vspace{0.2cm}\\
&& |\langle\hat{q}, i\xi
\hat{\theta}\rangle|\leq |\xi||\hat{\theta}||\hat{q}|\leq \varepsilon_3 |\xi|^2|\hat{\theta}|^2+C(\varepsilon_3)|\hat{q}|^2
\end{eqnarray*}
and
\begin{eqnarray*}
\left|-\eta|\xi|^2\func{Re} \langle i\xi(\hat{u}+\tau \hat{v}),  \hat{q}\rangle\right|
\leq\frac{ \varepsilon_4}{2} |\xi|^4| \hat{u}+\tau \hat{v}|^2+C(\varepsilon_4)|\xi|^2|\hat{q}|^2,\\
\left|\kappa |\xi|^4(\tau -\beta)\func{Re}(\hat{v}(\bar{\hat{u}}+\tau \bar{\hat{v}}))\right|
\leq \frac{ \varepsilon_4}{2} |\xi|^4|\hat{u}+\tau \hat{v}|^2+C(\varepsilon_4) |\xi|^4 | \hat{v}|^2
\end{eqnarray*}
where $\varepsilon_3$ and $\varepsilon_4$ are arbitrary positive constants. Plugging the above estimates into \eqref{dF_4_dt_1}, we get
\begin{eqnarray}\label{dF_4_dt_main}
&&\frac{d}{dt}\mathcal{F}_4(t)+(\kappa-\varepsilon_3)|\xi|^2  |\hat{\theta}|^2+(\kappa-\varepsilon_4)|\xi|^4|\hat{u}+\tau \hat{v}|^2\notag\\
&\leq& \kappa |\xi|^2|\hat{v}+\tau\hat{w}|^2+C(\varepsilon_3,\varepsilon_4)(1+|\xi|^2)|\hat{q}|^2+C(\varepsilon_4)|\xi|^4|\hat{v}|^2.
\end{eqnarray}
Now, define the functional $\LC(\xi,t)$ as
\begin{equation*}
\LC(\xi, t):=\gamma_0(1+|\xi| ^{2}+|\xi|^4)\hat{\EC}(\xi, t)+|\xi|^2\F_2(\xi, t) +\gamma_1 \mathcal{F}_4(\xi, t),  \label{Functional_L_Cattaneo}
\end{equation*}%
where $\gamma_0$ and $\gamma_1 $ are positive constants that
will be fixed later. Taking the derivative of $\LC(\xi ,t)$ with respect to $t$
and using (\ref{dE_dt_1_Cattaneo}), (\ref%
{dF_2_dt_1_Cattaneo}) and (\ref{dF_4_dt_main}), we obtain%
\begin{eqnarray*}\label{dK_dt_1}
&&\frac{d}{dt}\LC(\xi,t)+\Big((1-\varepsilon_1)-\gamma_1\kappa\Big)|\xi|^2|\hat{v}+\tau \hat{w}|^2\notag\\
&&+\Big(\gamma_1(\kappa-\varepsilon_4)-\varepsilon_2\Big)|\xi|^4(|\hat{u}+\tau \hat{v}|^2)
+\Big(\gamma_1(\kappa-\varepsilon_3)-\varepsilon_0\Big)|\xi|^2|\hat{\theta}|^2\notag\\
&&+\Big(\gamma_0(\beta-\tau )-C(\varepsilon_0,\varepsilon_1,\varepsilon_2,\varepsilon_4)\gamma_1 \Big)|\xi|^2(1+|\xi|^2+|\xi|^4)|\hat{v}|^2
\notag\\
&&+(\gamma_0-C(\varepsilon_3,\varepsilon_4)\gamma_1)(1+|\xi|^2+|\xi|^4)|\hat{q}|^2
\leq 0,
\end{eqnarray*}
where $\Lambda$ is a generic positive constant that depends on $\varepsilon_i$ for $i=1,\ldots,4$ and $\gamma_j$ for $j=0,1$, but independent on $t$  and $\xi$. We  have used the fact that $|\xi| ^{2}/(1+|\xi| ^{2})\leq |\xi |^{2}.$ Let us now fix the constants in the above estimate in such a way that all the coefficients in the previous inequality are strictly positive. First, we pick $\varepsilon_1$ and $\varepsilon_3$ and $\varepsilon_4$ small enough such that
\begin{equation*}
\varepsilon_1<1,\quad \varepsilon_3<\kappa,\quad\varepsilon_4<\kappa
\end{equation*}
After that, we choose $\gamma_1$ small enough such that
\begin{equation*}
\gamma_1<\frac{1-\varepsilon_1}{\kappa}.
\end{equation*}
Next, we choose $\varepsilon_0$ and $\varepsilon_2$ small enough such that
\begin{equation*}
\varepsilon_0<\gamma_1(\kappa-\varepsilon_3),\qquad \varepsilon_2<\gamma_1(\kappa-\varepsilon_4).
\end{equation*}
Finally and once all the above constants are fixed, we choose $\gamma_0$ large enough such that
\begin{equation*}
\gamma_0>\max\left\{\frac{C(\varepsilon_4)\gamma_1}{\beta-\tau},C(\varepsilon_3,\varepsilon_4)\gamma_1\right\}.
\end{equation*}
Consequently, we deduce that there exists a positive constant $\gamma_2$ such that for all $t\geq 0$, we have
\begin{eqnarray*}
&&\frac{d}{dt}\LC(\xi,t)+\gamma_2 |\xi|^2
\left\{|\hat{v}+\tau \hat{w}|^2+|\xi|^2|\hat{u}+\tau \hat{v}|^2+|\hat{\theta}|^2+|\xi|^2|\hat{v}|^2+|\hat{q}|^2 \right\}
\leq 0,
\end{eqnarray*}%
and, hence,
\begin{eqnarray}\label{K_Estimate_Main}
&&\frac{d}{dt}\LC(\xi,t)+\gamma_2 |\xi|^2 \hat{\EC}(t)
\leq 0.
\end{eqnarray}%
On the other hand using the Cauchy--Schwarz inequality  together with the Young inequality, we deduce that, for $\gamma_0$ large enough, there exist two positive constants $\gamma_3$ and $\gamma_4$ such that
\begin{equation}\label{Equivalence_E_L_Cattaneo}
\gamma_3(1+|\xi|^2+|\xi|^4)\hat{\EC}(\xi,t)\leq \LC(\xi,t)\leq \gamma_4(1+|\xi|^2+|\xi|^4)\hat{\EC}(\xi,t), \qquad \forall t\geq 0.
\end{equation}
Consequently, combining (\ref{Positivity_Energy_Cattaneo}), (\ref{K_Estimate_Main}) and (\ref{Equivalence_E_L_Cattaneo}) and using Gronwall's lemma, we deduce (\ref{V_Inequality}) because $\LC(\xi,t)$ and $|\VC(\xi,t)|^2$ are equivalent.
\end{proof}

\begin{proof}[Proof of Theorem \ref{Decay_Second_system_Cattaneo}]
The proof of Theorem \ref{Decay_Second_system_Cattaneo} can be done with the same method as in the proof of Theorem \ref{Decay_Second_system}. We can just see that the  behavior of the function $\rC(\xi)$ is as follows:
\begin{equation*}\label{varrho_behavior}
\rC(\xi)\geq\left\{
\begin{array}{ll}
\frac{1}{3}|\xi|^2,& \text{for } |\xi|\leq 1,\vspace{0.2cm} \\
\frac{1}{3}|\xi|^{-2}, & \text{for } |\xi|\geq 1,
\end{array}%
\right. .
\end{equation*}%
Consequently, the estimate of the low frequency part $(|\xi|\leq 1)$ can be done exactly as in the proof of Theorem \ref{Decay_Second_system}. On the other hand, using the estimate
\begin{equation*}\label{sup_inequality}
\sup_{|\xi|\geq 1} \left\{ \left\vert \xi \right\vert ^{-2\ell}e^{-c|\xi|
^{-2}t}\right\}\leq C(1+t)^{-\ell},
\end{equation*}
then the high frequency part $L_2$ corresponding to $|\xi|\geq 1$ is estimated as
\begin{eqnarray*}\label{L_2_estimate_C}
L_{2}
&\leq &C\sup_{|\xi|\geq 1} \left\{ \left\vert \xi \right\vert ^{-2\ell}e^{-c|\xi|^{-2} t}\right\}  \int_{\left\vert \xi \right\vert \geq 1}\left\vert \xi
\right\vert ^{2(k+\ell)}\vert \hat{\VC}(\xi ,0)\vert
^{2}d\xi\notag \\
&\leq &C(1+t)^{-\ell}\Vert \partial _{x}^{k+\ell}\VC^0\Vert _{L^{2}}^{2}.
\end{eqnarray*}
Collecting the estimates in the low and high frequency parts, we deduce (\ref{Main_estimate_Theorem_2}).
\end{proof}
\begin{remark}
The estimates in Theorems \ref{Decay_Second_system} and \ref{Decay_Second_system_Cattaneo} can be improved by allowing our initial data to be in some weighted spaces of $L^1(\R^N)$ where its total mass vanishes. See \cite{RaSa11_1,SaidKasi_2011} and \cite[Remark 3.2]{Racke_Said_2012_1}.
\end{remark}
\begin{remark}
Using the Gagliardo--Nirenberg inequality for $m>N/2$:
\begin{equation*}
\Vert u\Vert_{L^\infty (\R^N)}\leq C(N,m) \Vert u\Vert_{L^2(\R^N)}^{1-\frac{N}{2m}}\Vert \nabla^m u\Vert_{L^2(\R^N)}^{\frac{N}{2m}}
\end{equation*}
 then we can easily obtain the  decay estimates of the $L^\infty$-norm.
Hence, by  interpolation
inequalities, we obtain the  decay of the $L^p$-norms for $2\leq p\leq \infty$.
Such estimates for $1\leq p<2$ are not clear because we do not know any
bound for the $L^1$-norm of the solution.

\end{remark}

\subsubsection{Asymptotic behaviour of the eigenvalues}\label{Section_Asym_Cattaeo}

In this section, we compute the asymptotic expansion of the eigenvalues by using again the eigenvalues expansion method. This asymptotic behaviour will be in agreement with the decay rate seen in the previous section, as the following remark explains.
\begin{remark}\label{rmk:expansion3}
The decay rate in Theorem \ref{Decay_Second_system_Cattaneo} comes from the exponent
$$\rC(\xi)= \dfrac{|\xi|^2}{1+|\xi|^2+|\xi|^4}$$
of Proposition \ref{Main_Lemma_Cattaneo}. Observe that $\rC(\xi)\sim |\xi|^2$ when $|\xi|\to 0$, and $\rC(\xi)\sim |\xi|^{-2}$ when $|\xi|\to \infty$. This is the asymptotic behaviour we expect for the real parts of the slowest  eigenvalues in these cases (see Lemmas \ref{Eigenvalues_Expansion_Lemma} and \ref{Eigenvalues_Expansion_Lemma_2} below).
\end{remark}

The first thing we want is to compute the characteristic equation associated with \eqref{System_New_2_Cattaneo}. Because of the presence of inner products, this equation is not as easy to obtain as in the Fourier case. But we  proceed with some manipulations on this system in order to write it as a fifth order ODE, for which we will compute the corresponding characteristic equation straightforward.

First, by taking the dot product of equation \eqref{Fifth_equation_system__Fourier_Cattaneo} with $i\xi$, we get
\begin{equation}\label{Eq_5_1}
\tau_0 i\xi\cdot \hat{q}_t+i\xi\cdot\hat{q}-\kappa|\xi|^2  \hat{\theta}=0.
\end{equation}
Next, we take the time derivative of equation \eqref{Fourth_equation_system_Fourier_Cattaneo}, obtaining
\begin{equation}\label{Eq_4_1}
\hat{\theta} _{tt}+\kappa i\xi\cdot \hat{q}_t+\eta\tau|\xi|^2 \hat{w}_t+\eta |\xi|^2\hat{v}_t=0.
\end{equation}
Plugging \eqref{Eq_5_1} into \eqref{Eq_4_1} gives
\begin{equation}\label{Eq_4_2_1}
\hat{\theta} _{tt}-\frac{\kappa }{\tau_0}(i\xi\cdot \hat{q})+\frac{\kappa^2}{\tau_0}|\xi|^2\hat{\theta}+\eta\tau|\xi|^2 \hat{w}_t+\eta |\xi|^2\hat{v}_t=0.
\end{equation}
Using \eqref{Fourth_equation_system_Fourier_Cattaneo}, we obtain
\begin{equation}\label{Equ_4_3}
-\frac{\kappa}{\tau_0} i\xi\cdot \hat{q}=\frac{1}{\tau_0}\hat{\theta} _{t}+\frac{1}{\tau_0}\eta\tau|\xi|^2 \hat{w}+\frac{1}{\tau_0}\eta |\xi|^2\hat{v}.
\end{equation}
Inserting \eqref{Equ_4_3} into \eqref{Eq_4_2_1}, we get
\begin{equation}\label{Eq_4_2_2}
\hat{\theta} _{tt}+\frac{1}{\tau_0}\hat{\theta} _{t}+\frac{1}{\tau_0}\eta\tau|\xi|^2 \hat{w}+\frac{1}{\tau_0}\eta |\xi|^2\hat{v}+\frac{\kappa^2}{\tau_0}|\xi|^2\hat{\theta}+\eta\tau|\xi|^2 \hat{w}_t+\eta |\xi|^2\hat{v}_t=0.
\end{equation}
Now, taking the second derivative of \eqref{Third_equation_system_Fourier_Cattaneo}, with respect to $t$, we get

\begin{equation}\label{Equation_3_1}
\tau\hat{w}_{ttt}+\hat{w}_{tt}+|\xi|^2\hat{u}_{tt}+\beta |\xi|^2\hat{v}_{tt}-\eta|\xi|^2\hat{\theta}_{tt}=0.
\end{equation}
 Plugging \eqref{Eq_4_2_2} into \eqref{Equation_3_1}, we obtain
 \begin{eqnarray}\label{Eq_4_2_3}
&&\tau\hat{w}_{ttt}+\hat{w}_{tt}+|\xi|^2\hat{u}_{tt}+\beta |\xi|^2\hat{v}_{tt}+\frac{\eta}{\tau_0}|\xi|^2\hat{\theta} _{t}+\frac{\eta^2\tau}{\tau_0} |\xi|^4\hat{w}+\frac{\eta^2}{\tau_0} |\xi|^4\hat{v}\notag\\
&&+\frac{\eta\kappa^2}{\tau_0}|\xi|^4\hat{\theta}+\eta^2\tau|\xi|^4 \hat{w}_t+\eta^2 |\xi|^4\hat{v}_t=0.
\end{eqnarray}
Now, taking the derivative of \eqref{Third_equation_system_Fourier_Cattaneo} with respect to $t$, we get
\begin{equation}\label{Eq_3_3}
\eta|\xi|^2\hat{\theta}_t=\tau\hat{w}_{tt}+\hat{w}_t+|\xi|^2\hat{u}_t+\beta |\xi|^2\hat{v}_t
\end{equation}
Plugging \eqref{Eq_3_3} and \eqref{Third_equation_system_Fourier_Cattaneo} into \eqref{Eq_4_2_3}, we obtain
\begin{eqnarray}\label{Eq_4_2_4}
&&\tau\hat{w}_{ttt}+\hat{w}_{tt}+|\xi|^2\hat{u}_{tt}+\beta |\xi|^2\hat{v}_{tt}+\frac{\eta^2\tau}{\tau_0} |\xi|^4\hat{w}+\frac{\eta^2}{\tau_0} |\xi|^4\hat{v}\notag\\
&&+\eta^2\tau|\xi|^4 \hat{w}_t+\eta^2 |\xi|^4\hat{v}_t\notag\\
&&+\frac{1}{\tau_0}(\tau\hat{w}_{tt}+\hat{w}_t+|\xi|^2\hat{u}_t+\beta |\xi|^2\hat{v}_t
)\notag\\
&&+\frac{\kappa^2}{\tau_0}|\xi|^2(\tau\hat{w}_{t}+\hat{w}+|\xi|^2\hat{u}+\beta |\xi|^2\hat{v})=0 .
\end{eqnarray}
Now, we take the fourth  derivative of \eqref{First_equation_system_Fourier_Cattaneo} and the third derivative of \eqref{Second_equation_system_Fourier_Cattaneo}, we obtain, respectively
\begin{equation}\label{equ_5_u}
\frac{d^5}{dt^5} \hat{u}=\frac{d^4}{dt^4} \hat{v}=w_{ttt}.
\end{equation}
Now, plugging \eqref{equ_5_u} into \eqref{Eq_4_2_4}, we obtain
\begin{eqnarray*}
&&\tau \frac{d^5}{dt^5} \hat{u}+(1+\frac{\tau}{\tau_0}) \frac{d^4}{dt^4} \hat{u}+\left(\beta |\xi|^2+\frac{1}{\tau_0}+\eta^2\tau|\xi|^4+\frac{\kappa^2\tau}{\tau_0}|\xi|^2\right) \hat{u}_{ttt}\\
&&+\left(|\xi|^2+\frac{\eta^2\tau}{\tau_0} |\xi|^4+\frac{\kappa^2}{\tau_0}|\xi|^2+\eta^2 |\xi|^4+\frac{\beta}{\tau_0}|\xi|^2\right) \hat{u}_{tt}\\
&&+\left(\frac{\eta^2}{\tau_0} |\xi|^4+\frac{1}{\tau_0}|\xi|^2+\frac{\kappa^2\beta }{\tau_0}|\xi|^4\right)\hat{u}_t+\frac{\kappa^2}{\tau_0}|\xi|^4\hat{u}=0.
\end{eqnarray*}
So, we have written system \eqref{System_New_2_Cattaneo} as a fifth order ODE.  Now, if we name $\zeta=i|\xi|$, the characteristic equation associate to the above ODE (and hence to system \eqref{System_New_2_Cattaneo}) is
\begin{eqnarray}\label{charpol_Cattaneo}
&& \lambda^5+
\Big(\frac{1}{\tau_0}+\frac{1}{\tau}\Big)\lambda^4 +\frac{\tau\eta^2\tau_0\zeta^4-\left(\tau \kappa^2+\beta\tau_0\right)\zeta^2+1}{\tau\tau_0}\lambda^3  \\
\nonumber
&& +\frac{\left(\tau+\tau_0\right)\eta^2\zeta^2-\left(\beta+\tau_0+\kappa^2\right)}{\tau\tau_0}\,\zeta^2\lambda^2 +\frac{\left(\eta^2+\beta\kappa^2\right)\zeta^2-1}{\tau\tau_0}\,\zeta^2\lambda +\frac{\kappa^2\zeta^4}{\tau\tau_0}=0.
\end{eqnarray}
\begin{lemma}\label{Eigenvalues_Expansion_Lemma}
 Assume that $0<\tau<\beta$.  Then the  real parts of the eigenvalues of \eqref{System_New_2_Cattaneo} (i.e., the solutions of \eqref{charpol_Cattaneo}) satisfy for $|\xi|\rightarrow 0$ the following  asymptotic expansion:

\begin{eqnarray}\label{Exp_Near_0_Cattaneo}
\left\{
\begin{array}{ll}
  \func{Re}(\lambda_{1,2})(\xi) &= \dfrac{\tau-\beta}{2}|\xi|^2+O(|\xi|^3)\vspace{0.2cm} \\
  \func{Re}(\lambda_{3})(\xi) &= -\kappa^2|\xi|^2+O(|\xi|^3) \\
  \func{Re}(\lambda_{4})(\xi) &= -\dfrac{1}{\tau}+O(|\xi|^2)\vspace{0.2cm}\\
  \func{Re}(\lambda_{5})(\xi) &= -\dfrac{1}{\tau_0}+O(|\xi|^2).
  \end{array}
  \right.
\end{eqnarray}
\end{lemma}
\begin{proof}
Denoting the eigenvalues of the previous system as $\lambda_j(\zeta)$, $j=1,\ldots,5$, we can obtain their asymptotic expansions for $\zeta\to 0$:
\begin{eqnarray*}
  \lambda_{1,2}(\zeta) &=& \pm \zeta -\frac{\tau-\beta}{2}\zeta^2+O(\zeta^3), \\
  \lambda_{3}(\zeta) &=& \kappa^2 \zeta^2+O(\zeta^3), \\
  \lambda_{4}(\zeta) &=& -\frac{1}{\tau}+O(\zeta^2),\\
  \lambda_{5}(\zeta) &=& -\frac{1}{\tau_0}+O(\zeta^2).
\end{eqnarray*}

If we now come back to $\xi$, we have the asymptotic behavior \eqref{Exp_Near_0_Cattaneo} for the real parts of the previous eigenvalues when $|\xi|\to 0$.
Observe that when $0<\tau<\beta$ (dissipative case for the standard linear model) all the previous real parts are negative.
\end{proof}


We now proceed with the asymptotic behaviour of the eigenvalues when $|\xi|\to\infty$.
We have the following lemma.
\begin{lemma}\label{Eigenvalues_Expansion_Lemma_2}
 The  real parts of the eigenvalues of \eqref{System_New_2_Cattaneo} (i.e., the solutions of \eqref{charpol_Cattaneo}) satisfy for $|\xi|\rightarrow \infty$, the following  asymptotic expansion:
 \begin{equation*}\label{Asym_Expan_Infty_Cattaneo}
\left\{
\begin{array}{ll}
  \func{Re}(\lambda_{1,2})(\xi) &= -\dfrac{(\beta-\tau)\tau_0^2+\kappa^2\tau^2}{2\tau^2\eta^2\tau_0^2} |\xi|^{-2} +O(|\xi|^{-3}), \vspace{0.3cm}\\
  \func{Re}(\lambda_{3,4,5})(\xi) &= \sigma_{3,4,5} + O(|\xi|^{-1}),
\end{array}
\right.
\end{equation*}
 with $\sigma_{3,4,5}<0$.
\end{lemma}
\begin{proof}
 Again, taking $\nu=\zeta^{-1}=(i|\xi|)^{-1}$ in \eqref{charpol_Cattaneo}, the characteristic equation is now

\begin{eqnarray}\label{charpol_Cattaneo_infty}
&& \mu^5 +\Big(\frac{1}{\tau_0}+\frac{1}{\tau}\Big)\nu^2\mu^4 + \frac{\nu^4-(\tau \kappa^2 +\beta\tau_0)\nu^2+\tau\eta^2\tau_0}{\tau\tau_0}\mu^3 \\
&& + \frac{(\tau+\tau_0)\eta^2-(\beta+\tau_0+\kappa^2)\nu^2}{\tau\tau_0}\nu^2\mu^2 +\frac{(\eta^2+\beta\kappa^2)-\nu^2}{\tau\tau_0}\nu^4\mu +\frac{\kappa^2\nu^6}{\tau\tau_0}=0.
\nonumber
\end{eqnarray}
Again, observe that if $\lambda(\xi)$ is a solution of \eqref{charpol_Cattaneo}, then $$\mu(\nu)=\nu^2\lambda=\zeta^{-2}\lambda=-|\xi|^{-2}\lambda$$ is a solution of \eqref{charpol_Cattaneo_infty}. We can now compute the asymptotic expansion
$$\mu_j(\nu) =\mu_j^0+\mu_j^1\nu+\mu_j^2\nu^2+\mu_j^3\nu^3+\mu_j^4\nu^4+\ldots,\quad  j=1,\ldots,4$$of the roots of the characteristic polynomial \eqref{charpol_Cattaneo_infty} when $\nu\to 0$ (that is, when $|\xi|\to\infty$):

\begin{eqnarray*}
  \mu_{1,2}(\nu) &=& \pm \eta i +\frac{\tau \kappa^2+\beta\tau_0}{2\tau\eta\tau_0} i \nu^2 +\left( \frac{(\beta-\tau)\tau_0^2+\kappa^2\tau^2}{2\tau^2\eta^2\tau_0^2} \mp \frac{\beta^2\tau_0^2+6\beta\kappa^2\tau\tau_0+\kappa^4\tau^2}{8\tau^2\eta^3\tau_0^2} i\right)\nu^4 + O(\nu^5) \\
  \mu_{3,4,5}(\nu) &=&  \sigma_{3,4,5} \nu^2+O(\nu^3).
\end{eqnarray*}
(if $\eta,\tau,\tau_0\neq 0$) where $\sigma_{3,4,5}$ represent the three roots of the following third order equation in $\sigma$
\begin{equation}\label{sigma2}
\tau\eta^2\tau_0 \sigma^3 +\eta^2(\tau_0+\tau) \sigma^2+(\eta^2+\beta\kappa^2) \sigma+\kappa^2=0.
\end{equation}
Considering again the relation between $\mu$ and $\lambda$ and that $\nu=\zeta^{-1}=-i|\xi|^{-1}$, we recall that
$$\lambda_j(\xi) =-\mu_j^0|\xi|^2+\mu_j^1 i|\xi|+\mu_j^2-\mu_j^3 i|\xi|^{-1}-\mu_j^4 |\xi|^{-2}+O(|\xi|^{-3}),\quad j=1,\ldots,4.$$
So, we have that, when $|\xi|\to \infty$:
\begin{eqnarray*}
  \func{Re}(\lambda_{1,2})(\xi) &=& -\frac{(\beta-\tau)\tau_0^2+\kappa^2\tau^2}{2\tau^2\eta^2\tau_0^2} |\xi|^{-2} +O(|\xi|^{-3}) \\
\end{eqnarray*}
which is negative if $0<\tau<\beta$, but:
\begin{eqnarray*}
  \func{Re}(\lambda_{3,4,5})(\xi) &=& \sigma_{3,4,5} + O(|\xi|^{-1}).
\end{eqnarray*}
For the roots of \eqref{sigma2}, we consider
\begin{equation}\label{Equation_Third_Order}
f(X)=\tau \eta ^2 \tau_0  X^3+\eta ^2 (\tau_0+\tau )X^2 +\left(\eta ^2+\beta\right)X +1.
\end{equation}
It is well known that an algebraic equation of an odd degree with real coefficients has at least
one real root $\sigma_3$.  Now, in order to know the location of $\sigma_3$, we consider the equation  \eqref{Equation_Third_Order} with $X\in \R$.
Thus, we have
\begin{equation*}
f\left(-\frac{1}{\tau}-\frac{1}{\tau_0}\right)=
-\frac{ \left(\beta-\tau\right)\kappa^2\tau_0+\eta^2 (\tau+\tau_0 )+ \beta\kappa^2\tau }{\tau \tau_0 }<0,
\end{equation*}
since $0<\tau<\beta$.
Consequently, we obtain
\begin{equation*}
f\left(-\frac{1}{\tau}-\frac{1}{\tau_0}\right)f(0)<0.
\end{equation*}
Therefore,  equation \eqref{Equation_Third_Order} has at least one real root $X=\sigma_3$ in the interval $(-\frac{1}{\tau}-\frac{1}{\tau_0},0)$. In this case, we may write  \eqref{Equation_Third_Order} in the form
\begin{equation*}\label{Equation_Factor}
f(X)=(X-\sigma_3)\left(\tau \eta ^2 \tau_0 X^2+d_1 X+d_0\right)
\end{equation*}
with
\begin{equation*}
d_1=\tau \eta ^2 \tau_0\sigma_3+\eta ^2 (\tau+\tau_0 )\qquad \text{and}\qquad d_0= \beta\kappa^2 +\eta ^2 + d_1\sigma_3 = \beta\kappa^2 +\eta ^2 + \tau \eta ^2 \tau_0\sigma_3^2+\eta ^2 (\tau+\tau_0 )\sigma_3
\end{equation*}
or, also, $d_0=-\frac{1}{\sigma_3}$.
Now, let us denote by $\sigma_4$ and $\sigma_5$, the other two roots. Then, we have
\begin{equation*}
\sigma_4+\sigma_5=-\frac{d_1}{\tau \eta ^2 \tau_0}\qquad \textrm{and}\qquad \sigma_4\sigma_5=\frac{d_0}{\tau \eta ^2 \tau_0}.
\end{equation*}
Since  $\sigma_3\in(-\frac{1}{\tau}-\frac{1}{\tau_0},0)$ it is clear that $d_1>0$ and $d_0>0$.

Then if $\sigma_3$ and $\sigma_4$ are real, they are both negative (since their sum is negative and their product is positive).

If $\sigma_3$ and $\sigma_4$ are complex, then  they are conjugate and  therefore $$\func{Re}(\sigma_4)=\func{Re}(\sigma_5).$$
This implies that

\begin{eqnarray*}
\func{Re}(\sigma_4)&=&\func{Re}(\sigma_5)
= -\frac{1}{2} \left(\sigma_3+\frac{1}{\tau_0}+\frac{1}{\tau}\right)<0,
\end{eqnarray*}
since $\sigma_3\in (-\frac{1}{\tau}-\frac{1}{\tau_0},0).$
\end{proof}

\subsection{Stability results: the case $\beta=\tau$}\label{stabCattaneo_2}
\subsubsection{Lyapunov functional}
In this section as we did for the Fourier model, we show that the presence of the heat conduction allows us to push the above result to the case $\beta=\tau$. However, as in the Fourier model a slower decay rate is obtained. The main result here is given in the following theorem.

\begin{theorem}\label{thm:CattaneoDecay_tau=beta}
Assume that $\tau=\beta$ (with $\tau,\beta>0$).  Let $\WC(x,t)=(\tau u_{tt}+u_t,\nabla (\tau u_t+u),\theta,q)^T(x,t)$, where $(u,\theta,q)$ is the solution of (\ref{Bose_Problem}) when $\tau_0>0$. Let $s$ be a nonnegative integer and let $\WC^0=\WC(x,0)\in
H^{s}\left(
\mathbb{R}^N
\right) \cap L^{1}\left(
\mathbb{R}^N
\right) .$  Then, the following decay estimate
\begin{equation}\label{Main_estimate_Theorem_2}
\left\Vert \nabla^{k}\WC\left( t\right) \right\Vert _{L^{2}(\mathbb{R}^N)}\leq  C(1+t)^{-N/8-k/4}\left\Vert \WC^0 \right\Vert _{L^{1}(\mathbb{R}^N)}+ C(1+t)^{-\ell/3}\left\Vert \nabla^{k+\ell}\WC^0 \right\Vert _{L^{2}(\mathbb{R}^N)}
\end{equation}
holds,
for any $t\geq 0$ and $0\leq k+\ell\leq s$, where $C$ and $c$ are two positive constants independent of $t$ and $\WC^0$.
\end{theorem}

The proof of this theorem follows from the Proposition \ref{Main_Lemma_Cattaneo_2} below and the same techniques as in the proofs of Theorems \ref{Decay_Fourier_2} and \ref{Decay_Second_system_Cattaneo}. Hence, we omit its details here.

\begin{remark}
  Notice that the previous result shows the same exponent for the decay rate as the Fourier case when $\tau=\beta$ (see Theorem \ref{Decay_Fourier_2}) and also exhibits a regularity-loss phenomenon as in the Cattaneo case with $\tau<\beta$, although the loss of regularity is not the same, (see Theorem \ref{Decay_Second_system_Cattaneo}). We recall that this regularity loss part comes from the asymptotic behaviour of the eigenvalues when $|\xi|\to \infty$. And, as we will see in Remark \ref{rmk:decayCatt_tau=beta} in next section, this asymptotic behaviour is not the same as the one obtained for
  $\vrC(\xi)$ in Proposition \ref{Main_Lemma_Cattaneo_2} below (which, actually, is the exponent used to prove Theorem \ref{thm:CattaneoDecay_tau=beta}). Hence, unlike the previous sections, the exponent obtained in this proposition (and, hence, the regularity loss phenomenon in Theorem \ref{thm:CattaneoDecay_tau=beta}) is true, but may not optimal.
\end{remark}
As we said, the main result to prove the previous decay result is the following proposition.
\begin{proposition}\label{Main_Lemma_Cattaneo_2}
Assume that $\tau=\beta$. Let $\WC(x,t)=(\tau u_{tt}+u_t,\nabla (\tau u_t+u),\theta,q)^T(x,t)$, where $(u(x,t),\theta(x,t),q(x,t))$ is the solution of \eqref{Bose_Problem} for $\tau_0>0$. Then we have
\begin{equation*}  \label{V_Inequality_2}
|\hat{\WC}(\xi,t)|^2\leq Ce^{-c\vrC(\xi)t}|\hat{\WC}(\xi,0)|^2,
\end{equation*}
where
\begin{equation*}
\vrC(\xi)=\frac{|\xi|^4}{(1+|\xi|^2)^2(1+|\xi|^2+|\xi|^4+|\xi|^6)},
\end{equation*}
and $C$ and $c$ are two positive constants.
\end{proposition}

\begin{proof}
Now, for $\beta=\tau$, the energy functional \eqref{Energy_functional_Cattaneo} takes the form
\begin{equation*}\label{Energy_functional_Cattaneo_2}
\hat{\EcC}(\xi,t):=\frac{1}{2}\left\{ | \hat{v}+\tau\hat{w}|^2+|\xi|^2|\hat{u}+\tau\hat{v}|^2+|\hat{\theta}|^2+\tau_0|\hat{q}|^2\right\}
\end{equation*}
and satisfies for all $t\geq 0$,
\begin{equation}  \label{dE_dt_1_Cattaneo_2}
\frac{d}{dt}\hat{\EcC}(\xi, t)=-|\hat{q}|^2.
\end{equation}
The estimate \eqref{dF_1_dt_tau_2} also holds for the Cattaneo model.

Now, to get a dissipative term for $|\hat{v}+\tau\hat{w}|^2$, we rewrite system \eqref{System_New_2_Cattaneo} as
\begin{eqnarray}\label{Cattaneo_New_Form}
\left\{
\begin{array}{ll}
(\hat{v}+\tau\hat{w})_t+|\xi|^2 (\hat{u}+\tau  \hat{v})-\eta |\xi|^2 \hat{\theta} =0,\vspace{0.2cm}\\
\hat{\theta} _{t}+\kappa i\xi\cdot \hat{q}+\eta|\xi|^2  (\hat{v}+\tau\hat{w})=0,
\vspace{0.2cm}\\
 \tau_0\hat{q}_t+\hat{q}+i\xi \kappa \hat{\theta}=0.
\end{array}
\right.
\end{eqnarray}
Hence, we define the functional $\F_3(\xi,t)$ as in \eqref{F_3_Equation} and obtain, by taking its time derivative,
\begin{eqnarray}\label{dF_3_dt_Catt_3}
&&\frac{d}{dt}\F_3(\xi,t)+\eta|\xi|^2|\hat{v}+\tau\hat{w}|^2\notag\\
&=&-\kappa\func{Re}\langle i\xi \cdot \hat{q},\hat{v}+\tau \hat{w}%
\rangle +\eta|\xi|^2|\hat{\theta}|^2-|\xi|^2\func{Re}(\bar{\hat{\theta}}(\hat{u}+\tau\hat{v})).
\end{eqnarray}
Next, taking the dot product of the third equation in %
\eqref{Cattaneo_New_Form} with $i\xi \eta (\overline{\hat{u}+\tau \hat{v}})$, we get
\[
\tau _{0}\eta \langle i\xi \cdot \hat{q}_{t},\hat{u}+\tau \hat{v}\rangle
+\eta \langle i\xi \cdot \hat{q},\hat{u}+\tau \hat{v}\rangle -|\xi
|^{2}\kappa \eta \hat{\theta}(\bar{\hat{u}}+\tau \bar{\hat{v}})=0
\]%
adding and subtracting $\eta \tau _{0}\langle (\hat{u}+\tau \hat{v}%
)_{t},i\xi \cdot \hat{q}\rangle =\eta \tau _{0}\langle \hat{v}+\tau \hat{w}%
,i\xi \cdot \hat{q}\rangle $ to the above equation and taking the real part,
we obtain
\begin{eqnarray}
&&\tau _{0}\eta \frac{d}{dt}\func{Re}\langle i\xi (\hat{u}+\tau \hat{v})),%
\hat{q}\rangle   \nonumber  \label{F_q_4} \\
&=&\eta \tau _{0}\func{Re}\langle i\xi (\hat{v}+\tau \hat{w})),\hat{q}%
\rangle +\eta \func{Re}(\langle i\xi \cdot \hat{q},\hat{u}+\tau \hat{v}%
\rangle )-|\xi |^{2}\kappa \eta \func{Re}(\hat{\theta}(\bar{\hat{u}}+\tau
\bar{\hat{v}})).
\end{eqnarray}
We define the functional
\begin{eqnarray*}
\tilde{\F}_3(\xi,t)=|\xi|^2\Big(\F_3(\xi,t)-\frac{\tau_0}{\kappa}\func{Re}\langle i\xi (\hat{u}+\tau \hat{v}),%
\hat{q}\rangle\Big).
\end{eqnarray*}
Hence, from \eqref{dF_3_dt_Catt_3} and \eqref{F_q_4}, we deduce that
\begin{eqnarray*}
&&\frac{d}{dt}\tilde{\F}_{3}(\xi ,t)+\eta |\xi |^{4}|\hat{v}+\tau \hat{w}|^{2}
\nonumber  \label{dF_3__tilde_dt_Catt_3} \\
&=&-\kappa |\xi |^{2}\func{Re}\langle i\xi \cdot \hat{q},\hat{v}+\tau \hat{w}%
\rangle +\eta |\xi |^{4}|\hat{\theta}|^{2}  \nonumber \\
&&-\frac{\tau _{0}}{\kappa }|\xi |^{2}\func{Re}\langle i\xi (\hat{v}+\tau
\hat{w})),\hat{q}\rangle -\frac{1}{\kappa }|\xi |^{2}\func{Re}(\langle i\xi
\cdot \hat{q},\hat{u}+\tau \hat{v}\rangle )\\
&=&  \left(\frac{\tau_0}{\kappa}-\kappa\right) |\xi |^{2}\func{Re}\langle i\xi \cdot \hat{q},\hat{v}+\tau \hat{w}%
\rangle +\eta |\xi |^{4}|\hat{\theta}|^{2} +\frac{1}{\kappa }|\xi |^{2}\func{Re}(\langle
 \hat{q},i\xi(\hat{u}+\tau \hat{v})\rangle ). \nonumber
\end{eqnarray*}
Applying Young's inequality, we obtain
\begin{eqnarray}\label{F_3_Estimate_Cattaneo}
&&\frac{d}{dt}\tilde{\F}_3(\xi,t)+(\eta-\tilde{\varepsilon}_1)|\xi|^4|\hat{v}+\tau\hat{w}|^2\notag\\
&\leq & \tilde{\varepsilon}_2|\xi|^6|\hat{u}+\tau \hat{v}|^2+\eta|\xi|^4|\hat{\theta}|^2+C(\tilde{\varepsilon}_1, \tilde{\varepsilon}_2)(1+|\xi|^2)|\hat{q}|^2.
\end{eqnarray}
Now, define the functional
\begin{eqnarray*}
\tilde{\F}_4(\xi,t)=\mathcal{F}_3(\xi,t)-|\xi|^2\tau _{0}\eta \func{Re}\langle i\xi (\hat{u}+\tau \hat{v})),%
\hat{q}\rangle.
\end{eqnarray*}
Then, we obtain from \eqref{dF_3_dt_3} and \eqref{F_q_4}
\begin{eqnarray*}\label{dF_4_dt_tilde}
\frac{d}{dt}\tilde{\F}_4(\xi,t)+\kappa|\xi|^2  |\hat{\theta}|^2&=&\tau_0\kappa |\xi\cdot \hat{q}|^2+\func{Re}\langle\hat{q}, i\xi %
\hat{\theta}\rangle
 -\eta\tau_0|\xi|^2\func{Re}\langle i\xi(\hat{v}+\tau\hat{w}),\hat{q}\rangle .
\end{eqnarray*}
Applying Young's inequality, we find for any $\tilde{\varepsilon}_3, \tilde{\varepsilon}_4>0$,
\begin{eqnarray}\label{dF_4_dt_tilde_Ineq}
&&\frac{d}{dt}\tilde{\F}_4(\xi,t)+(\kappa-\tilde{\varepsilon}_3)|\xi|^2  |\hat{\theta}|^2\notag\\
&\leq &\tilde{\varepsilon}_4\frac{|\xi|^2}{1+|\xi|^2}|\hat{v}+\tau \hat{w}|^2+C(\tilde{\varepsilon}_3,\tilde{\varepsilon}_4)(1+|\xi|^2+|\xi|^4+|\xi|^6)|\hat{q}|^2.
\end{eqnarray}
Now, we define the Lyapunov functional as
\begin{eqnarray} \label{Lyap_Cattanoe_2}
\cLC(\xi,t)&=& N_0(1+|\xi|^2+|\xi|^4+|\xi|^6)\hat{\EcC}(\xi,t)+N_1 \frac{|\xi|^4}{(1+|\xi|^2)^2}\F_1(\xi,t)\notag\\&&+N_2\frac{1}{(1+|\xi|^2)^2}\tilde{\F}_3(\xi,t)+N_3\frac{|\xi|^2}{1+|\xi|^2}\tilde{\F}_4(\xi,t),
\end{eqnarray}
where $N_i,\, i=0,\dots,3$ are positive constants that should be fixed later on.
Hence, taking the time derivative of \eqref{Lyap_Cattanoe_2} and making use of \eqref{dF_1_dt_tau_2}, \eqref{dE_dt_1_Cattaneo_2},  \eqref{F_3_Estimate_Cattaneo} and \eqref{dF_4_dt_tilde_Ineq} (and several inequalities on the fractions and powers of $|\xi|$ involved), we get
\begin{eqnarray}\label{d_L_dt_Cattaneo_2}
\left.
\begin{array}{l}
\dfrac{d}{dt}\cLC(\xi ,t)+\left[ N_{2}(\eta -\tilde{\varepsilon}%
_{1})-N_{1}-N_{3}\tilde{%
\varepsilon}_{4}\right] \dfrac{|\xi |^{4}}{(1+|\xi |^{2})^2}|\hat{v}+\tau \hat{w}|^{2}%
\vspace{0.2cm} \\
+\left[ N_{1}(1-\epsilon _{0})-N_{2}\tilde{\varepsilon}_{2}\right] \dfrac{|\xi |^{6}}{(1+|\xi |^{2})^2}|\hat{u}+\tau \hat{v}%
|^{2}\vspace{0.2cm} \\
+\left[ N_{3}(\kappa -\tilde{\varepsilon}_{3})-N_{2}\eta -C\left( \epsilon
_{0}\right) N_{1}\right] \dfrac{|\xi |^{4}}{(1+|\xi|^2)^2}|\hat{\theta}|^{2} \vspace{0.2cm}\\
+[N_{0}-\tilde{\Lambda}]\left( 1+|\xi |^{2}+|\xi |^{4}+|\xi |^{6}\right) |%
\hat{q}|^{2}\leq 0,\qquad \forall t\geq 0,%
\end{array}%
\right.
\end{eqnarray}
where $\tilde{\Lambda}$ is a generic positive constants that depends on $\epsilon_0,\dots, N_1,\dots$, but it does not depend  on $N_0$. We fix the above constants as follows: we take $\epsilon_0,\tilde{\varepsilon}_1$ and $\tilde{\varepsilon}_3$ small enough such that
\begin{eqnarray*}
\epsilon_0<1,\quad\tilde{\varepsilon}%
_{1}<\eta,\quad \text{and}\quad \tilde{\varepsilon}%
_{3}<\kappa.
\end{eqnarray*}
Now, we fix $N_1=1$ and choose $N_2$ large enough such that
\begin{eqnarray*}
N_{2}>\frac{1}{\eta -\tilde{\varepsilon}%
_{1}}.
\end{eqnarray*}
 Then, we fix $N_3$ large enough such that
 \begin{eqnarray*}
 N_{3}(\kappa -\tilde{\varepsilon}_{3})>N_{2}\eta +C\left( \epsilon
_{0}\right).
\end{eqnarray*}
After that we pick $\tilde{\varepsilon}%
_{2}$ and  $\tilde{\varepsilon}%
_{4}$ small enough such that
 \begin{eqnarray*}
\tilde{\varepsilon}%
_{2}<\frac{1-\epsilon _{0}}{N_2}\qquad \text{and}\qquad\tilde{\varepsilon}%
_{4}<\frac{N_{2}(\eta -\tilde{\varepsilon}%
_{1})-1}{N_3}
\end{eqnarray*}
Once all the above constants are fixed, we take $N_0$ large enough such that $N_0>\tilde{\Lambda}$.
Consequently \eqref{d_L_dt_Cattaneo_2} becomes
\begin{eqnarray}\label{L_Main_alpha}
\dfrac{d}{dt}\cLC(\xi ,t)+\tilde{\alpha}\dfrac{|\xi |^{4}}{(1+|\xi |^{2})^2} \hat{\EcC}(\xi, t)\leq 0,\qquad \forall t\geq 0.
\end{eqnarray}
On the other hand, it is not hard to see that for all $t\geq 0$,
$$\cLC(\xi ,t)\sim (1+|\xi|^2+|\xi|^4+|\xi|^6)\hat{\EcC}(\xi, t). $$
Consequently, this together with \eqref{L_Main_alpha} leads to our desired result.
\end{proof}


\subsubsection{Eigenvalues expansion}
In this section, we compute the asymptotic expansion of the eigenvalues in the Cattaneo model when $\tau=\beta$ using, as before, the eigenvalues expansion method. In this case, and as a difference from the previous sections, this asymptotic behaviour will not be the same as the one that in the previous section gives us the decay and regularity of the norm related to the solution, as the following remark explains.
\begin{remark}\label{rmk:decayCatt_tau=beta}
The decay rate in Theorem \ref{thm:CattaneoDecay_tau=beta} comes from the exponent
$$\vrC(\xi)=\frac{|\xi|^4}{(1+|\xi|^2)^2(1+|\xi|^2+|\xi|^4+|\xi|^6)}, $$
of Proposition \ref{Main_Lemma_Cattaneo_2}. Observe that $\vrC(\xi)\sim |\xi|^4$ when $|\xi|\to 0$, and $\vrC(\xi)\sim |\xi|^{-6} $ when $|\xi|\to \infty$. From Lemmas \ref{lemma:eig_0_Catt_2} and \ref{lemma:eig_infty_Catt_2} below, we can see that the asymptotic behaviour of the real parts of the slowest  eigenvalues when $|\xi|\to 0$ is also $|\xi|^4$. But when $|\xi|\to\infty$ the slowest real parts of the eigenvalues behave as $|\xi|^{-2}$, which is a different rate than the one obtained in Proposition \ref{Main_Lemma_Cattaneo_2}. That allows us to conclude that the decay exponent obtained in this Proposition \ref{Main_Lemma_Cattaneo_2} may not optimal. And, hence, the decay result of Theorem \ref{thm:CattaneoDecay_tau=beta} may not optimal either. More concretely, while its decay rate is in agreement with the present results (as for $|\xi|\to 0$ both asymptotic behaviours agree), the regularity-loss  of the solutions (which comes from the asymptotic behaviour when $|\xi|\to\infty$) may be improved.
\end{remark}

First, observe that when $\tau=\beta$ the characteristic equation  \eqref{charpol_Cattaneo} takes the form
\begin{eqnarray}\label{charpol_Cattaneo_2}
&& \lambda^5+
\left(\frac{1}{\tau_0}+\frac{1}{\tau}\right)\lambda^4 +\frac{\eta^2\tau_0\zeta^4-\left( \kappa^2+\tau_0\right)\zeta^2+1}{\tau_0}\lambda^3  \\
\nonumber
&& +\frac{\left(\tau+\tau_0\right)\eta^2\zeta^2-\left(\tau+\tau_0+\kappa^2\right)}{\tau\tau_0}\,\zeta^2\lambda^2 +\frac{\left(\eta^2+\tau\kappa^2\right)\zeta^2-1}{\tau\tau_0}\,\zeta^2\lambda +\frac{\kappa^2\zeta^4}{\tau\tau_0}=0
\end{eqnarray}
where $\zeta=i|\xi|$.

\begin{lemma}\label{lemma:eig_0_Catt_2}
We assume that $0<\tau=\beta$. In this case, we get the following expansion for $|\xi|\rightarrow 0$:
\begin{eqnarray*}\label{Exp_Near_0_Cattaneo_2}
\left\{
\begin{array}{ll}
  \func{Re}(\lambda_{1,2})(\xi) &= -\dfrac{\eta^2\kappa^2}{2}|\xi|^4+O(|\xi|^5),\vspace{0.2cm} \\
  \func{Re}(\lambda_{3})(\xi) &= -\kappa^2|\xi|^2+O(|\xi|^3), \vspace{0.2cm} \\
  \func{Re}(\lambda_{4})(\xi) &= -\dfrac{1}{\tau}+O(|\xi|^2),\vspace{0.2cm}\\
  \func{Re}(\lambda_{5})(\xi) &= -\dfrac{1}{\tau_0}+O(|\xi|^2).
  \end{array}
  \right.
\end{eqnarray*}
In particular, all these real parts are negative.
\end{lemma}
\begin{pf}
As in the previous sections, it is easy to compute that for $\zeta\rightarrow 0$, we have the following asymptotic expansion for the roots of \eqref{charpol_Cattaneo_2}:
\begin{eqnarray*}
  \lambda_{1,2}(\zeta) &=& \pm \zeta \pm \left(\frac{\eta^2}{2}+\frac{5\tau^2}{4}\right)\zeta^3-\frac{\eta^2\kappa^2}{2}\zeta^4+O(\zeta^5), \\
  \lambda_{3}(\zeta) &=& \kappa^2 \zeta^2+O(\zeta^3), \\
  \lambda_{4}(\zeta) &=& -\frac{1}{\tau}+O(\zeta^2),\\
  \lambda_{5}(\zeta) &=& -\frac{1}{\tau_0}+O(\zeta^2).
\end{eqnarray*}
Now, coming back to the variable $\xi$, we get the following expansion for $|\xi|\rightarrow 0$:
\begin{eqnarray*}\label{Exp_Near_0_Cattaneo_2}
\left\{
\begin{array}{ll}
  \func{Re}(\lambda_{1,2})(\xi) &= -\dfrac{\eta^2\kappa^2}{2}|\xi|^4+O(|\xi|^5),\vspace{0.2cm} \\
  \func{Re}(\lambda_{3})(\xi) &= -\kappa^2|\xi|^2+O(|\xi|^3), \vspace{0.2cm}\\
  \func{Re}(\lambda_{4})(\xi) &= -\dfrac{1}{\tau}+O(|\xi|^2),\vspace{0.2cm}\\
  \func{Re}(\lambda_{5})(\xi) &= -\dfrac{1}{\tau_0}+O(|\xi|^2),
  \end{array}
  \right.
\end{eqnarray*}
which, indeed, are negative.
\end{pf}

\begin{lemma}\label{lemma:eig_infty_Catt_2}
We assume that $0<\tau=\beta$. In this case, we get the following expansion for $|\xi|\rightarrow \infty$:
\begin{equation*}\label{Asym_Expan_Infty_Cattaneo_2}
\left\{
\begin{array}{ll}
  \func{Re}(\lambda_{1,2})(\xi) &= -\dfrac{\kappa^2\tau^2}{2\tau^2\eta^2\tau_0^2} |\xi|^{-2} +O(|\xi|^{-3}), \vspace{0.3cm}\\
  \func{Re}(\lambda_{3,4,5})(\xi) &= \tilde{\sigma}_{3,4,5} + O(|\xi|^{-1}),
\end{array}
\right.
\end{equation*}
with $\tilde{\sigma}_{3,4,5}<0$, as the rest of the real parts.
\end{lemma}
\begin{pf}
We proceed as in the analogous lemmas of the previous sections (see, for instance, Lemma \ref{Eigenvalues_Expansion_Lemma}). Hence, applying the change of variables $\nu=\zeta^{-1}=(i|\xi|)^{-1}$, computing the asymptotic expansion of $\mu(\nu)$ (the roots of the corresponding characteristic polynomial for $\nu$) when $\nu\to 0$, and coming back to the eigenvalues $\lambda(\xi)$, we can easily show that the following expansion holds
\begin{equation*}\label{Asym_Expan_Infty_Cattaneo_2}
\left\{
\begin{array}{ll}
  \func{Re}(\lambda_{1,2})(\xi) &= -\dfrac{\kappa^2\tau^2}{2\tau^2\eta^2\tau_0^2} |\xi|^{-2} +O(|\xi|^{-3}), \vspace{0.3cm}\\
  \func{Re}(\lambda_{3,4,5})(\xi) &= \tilde{\sigma}_{3,4,5} + O(|\xi|^{-1}),
\end{array}
\right.
\end{equation*}
with $\tilde{\sigma}_{3,4,5}$ being the roots of \eqref{Equation_Third_Order} when $\tau=\beta$. In this case, these roots can be explicitly calculated and are 
$$\tilde{\sigma_3}=-\dfrac{1}{\tau},\qquad \tilde{\sigma_{4,5}}=-\dfrac{\eta \pm \sqrt{\eta^2-4\kappa^2\tau_0}}{2\eta\tau_0}.$$
Hence, we can conclude all the eigenvalues have negative real parts in this case too.
\end{pf}

\section*{Acknowledgements}
M. Pellicer is part of the Catalan Research group 2017 SGR 1392 and has been supported by the MINECO grant MTM2017-84214-C2-2-P (Spain), and also by MPC UdG 2016/047 (U. of Girona, Catalonia).


\end{document}